\documentclass[12pt,leqno]{amsart}
\usepackage{amsmath}
\usepackage{amsthm}
\usepackage{amssymb}
\def\overset#1#2{{\mathrel{\mathop {{#2}_{}}\limits^{#1}}}}
\def\underset#1#2{{\mathrel{\mathop {{}_{} {#2}}\limits_{{#1}_{}}}}}
\def\upplim_#1{\underset{#1}{\overline\lim}\;}
\def\lowlim_#1{\underset{#1}{\underline\lim}\;}
\newtheorem{corollary}[equation]{Corollary}
\newtheorem{claim}[equation]{\indent \rm {\it Claim}}
\newtheorem{definition}[equation]{Definition}

\newtheorem{lemma}[equation]{Lemma}

\newtheorem{proposition}[equation]{Proposition}

\newtheorem{theorem}[equation]{Theorem}

\newcommand{\C}{{\mathbb{C}}}

\newcommand{\N}{\mathbb{N}}
\newcommand{\B}{\mathbb{B}}

\renewcommand{\P}{{\mathbb{P}}}

\newcommand{\R}{{\mathbb{R}}}

\newcommand{\supp}{\mathrm{Supp}\,}

\newcommand{\Z}{\mathbb{Z}}
\numberwithin{equation}{section}

\begin{document}
\title[Meromorphic mappings of a complete connected K\"{a}hler manifold]{Meromorphic mappings of a complete connected K\"{a}hler manifold into a projective space sharing hyperplanes} 

\author{Si Duc Quang$^{1,2}$}
\address{$^{1}$Department of Mathematics, Hanoi National University of Education\\
136-Xuan Thuy, Cau Giay, Hanoi, Vienam.}
\address{$^{2}$Thang Long Institute of Mathematics and Applied Sciences\\
Nghiem Xuan Yem, Hoang Mai, Ha Noi.}
\email{quangsd@hnue.edu.vn}

\def\thefootnote{\empty}

\begin{abstract} {Let $M$ be a complete K\"{a}hler manifold, whose universal covering is biholomorphic to a ball $\B^m(R_0)$ in $\C^m$\ ($0<R_0\le +\infty$). In this article, we will show that if three meromorphic mappings $f^1,f^2,f^3$ of $M$ into $\P^n(\C)\ (n\ge 2)$ satisfying the condition $(C_\rho)$  and sharing $q\ (q> 2n+1+\alpha+\rho K)$ hyperplanes in general position regardless of multiplicity with certain positive constants $K$ and $\alpha <1$ (explicitly estimated), then $f^1=f^2$ or $f^2=f^3$ or $f^3=f^1$. Moreover, if the above three mappings share the hyperplanes with mutiplicity counted to level $n+1$ then $f^1=f^2=f^3.$ Our results generalize the finiteness and uniqueness theorems for meromorphic mappings of $\C^m$ into $\P^n(\C)$ sharing $2n+2$ hyperplanes in general position with truncated multiplicity.}
\end{abstract}
\maketitle

\section{Introduction}

In $1926$, R. Nevanlinna \cite{N} showed that there are at most two distinct non-constant meromorphic functions $f$ and $g$ on the complex plane $\C$ having the same inverse images for four distinct values, and these functions must be linked by a M\"{o}bius transformation. This result is usually called the four values theorem of Nevanlinna. After that, many authors have  extended and improved the result of Nevanlinna to the case of meromorphic mappings into complex projective spaces. These theorems are called finiteness theorems. Firstly, in 1983, L. Smiley \cite{Smiley} showed that there are at most two distinct linearly non-degenerate meromorphic mappings from $\C^m$ into $\P^n(\C)$ sharing $3n+1$ hyperplanes in general position regardless of multiplicity. Here, two meromorphic mappings are said to share a hyperplane if they have the same inverse image for that hyperplane and they coincide on this inverse image. The best result on this problem is recently given by the author \cite{SDQ12,SDQ19} when we reduced the number $3n+1$ of hyperplanes in the result of L. Smiley to $2n+2$. To state this result, first of all we recall the following.

Let $\B^m(R_0)$ be the ball $\{z\in\C^m;  \|z\|<R_0\}$, where $0<R_0\le +\infty$. Let $f$ be a non-constant meromorphic mapping of $\B^m(R_0)$ into $\P^n(\C)$ with a reduced representation $f = (f_0 : \dots : f_n)$, and $H$ be a hyperplane in $\P^n(\C)$ given by $H=\{a_0\omega_0+\cdots +a_n\omega_n=0\},$ where $(a_0,\ldots ,a_n)\ne (0,\ldots ,0)$.
Set  $(f,H)=\sum_{i=0}^na_if_i$. We see that $\nu_{(f,H)}$ is the pull-back divisor of $H$ by $f$, which is defined independently from the choice of the reduced representation of $f$ and the representation of $H$. 

Let $H_1,\ldots ,H_q$ be $q$ hyperplanes  of $\P^n(\C)$ in general position.
Assume that $f$ is linearly non-degenerate and satisfies
$$\dim f^{-1}(H_i)\cap f^{-1}(H_j) \le m-2 \quad (1 \le i<j \le q).$$ 
Let $d$ be a positive integer. We consider the set $\mathcal {F}(f,\{H_i\}_{i=1}^q,d)$ of all meromorphic mappings 
$g: \B^m(R_0) \to \P^n(\C)$ satisfying the following conditions:
\begin{itemize}
\item[(a)] $\nu_ {(f,H_i)}^{[d]}=\nu_{(g,H_i)}^{[d]}\quad (1\le i \le q),$
\item[(b)] $f(z) = g(z)$ on $\bigcup_{i=1}^{q}f^{-1}(H_i)$.
\end{itemize}
Here, by $\nu_{\varphi}$ we denote the divisor of the meromorphic function $\varphi$ and $\nu^{[d]}_{\varphi}=\min\{\nu_{\varphi},d\}$.
\vskip0.2cm
For the case of $R_0=+\infty$, the best finiteness theorem available at the present is stated as follows.

\vskip0.2cm
\noindent
{\bf Theorem A (\cite[Theorem 1.1]{SDQ19}).} {\it \ If $n\ge 2$ and $q=2n+2$ then $\sharp\ \mathcal {F}(f,\{H_i\}_{i=1}^q,1)\le 2.$}

\vskip0.2cm
We would also like to emphasize here, Theorem A is a weak form of  \cite[Theorem 1.1]{SDQ19}. Actually, in \cite[Theorem 1.1]{SDQ19} all zeros of functions $(f,H_i)$ with multiplicity more than a certain number are omitted in the sharing hyperplanes condition. 

Our first purpose in this paper is to generalize Theorem A to the case, where the meromorphic mapping $f$ is from a complete K\"{a}hler manifold into $\P^n(\C)$. We would like to emphazise here that, in order to study the finiteness problem of meromorphic mappings for the case of mappings from $\C^m,$  almost all authors use Cartan's auxialiary functions (see Definition \ref{2.2}) and compare the counting functions of these auxialiary functions with the characteristic functions of the mappings. However, in the general case of K\"{a}hler manifold, this method may do not work since this comparation does not make sense if the growth of the characteristic functions do not increase quickly enough. In order to overcome this difficulty, in this paper, we will introduce the notion of small integration and bounded integration for plurisubharmonic functions with respect to a set of meromorphic mappings (see Definitions \ref{3.1} and \ref{3.3}). Our essential key in the proof of the main results of this paper is Proposition \ref{3.5}, which can be considered as a general form of finiteness theorem for meromorphic mampings on K\"{a}hler manifold. Our method in this paper is not only used to study finiteness problem of meromorphic mapping, but also may be applied to study unicity, degeneracy and algebraic dependence problems of meromorphic mappings. Many results for the case of meromorphic mappings on $\C^m$ can be translated to the case of mappings on K\"{a}hler manifold by this method.

To state our first main result, we need to recall the following.

Let $M$ be an $m$-dimensional connected K\"{a}hler manifold with K\"{a}hler form $\omega$ and $f$ be a meromorphic map of $M$ into $\P^n(\C)$. Throughout this paper, we always assume that the universal covering of $M$ is biholomorphic to a ball $\B^m(R_0)$ in $\C^m$\ ($0<R_0\le +\infty$). For $\rho>0$, we
say that $f$ satisfies the condition ($C_\rho$) if there exists a nonzero bounded continuous real-valued function $h$ on $M$ such that
$$\rho\Omega_f+dd^c\log h^2\ge \mathrm{Ric}\omega,$$
where $\Omega_f$ denotes the pull-back of the Fubini-Study metric form on $\P^n(\C)$ by $f$.

Let $f$ be a linearly non-degenerate meromorphic mapping from $M$ into $\P^n(\C)$ which satisfies the condition $(C_\rho)$.
Let $H_1,\ldots,H_q$ be $q$ hyperplanes of $\P^n(\C)$ in general possition. 
Denote by $\nu_{(f,H_i)}$ the pull-back divisor of $H_i$ by $f$. 
Assume that 
$$\dim f^{-1}(H_i)\cap f^{-1}(H_j) \le m-2 \quad (1 \le i<j \le q).$$ 
The family $\mathcal {F}(f,\{H_i\}_{i=1}^q,d)$ is defined similarly as above.

In this paper, we will prove the following finiteness theorem for meromorphic mappings from K\"{a}hler manifold into $\P^n(\C)$ sharing hyperplanes regardless of multiplicity as follows.

\begin{theorem}\label{1.1}
Let $M$ be an $m$-dimensional connected K\"{a}hler manifold whose universal covering is biholomorphic to is biholomorphic to a ball $\B^m(R_0)$ in $\C^m$\ ($0<R_0\le +\infty$), and let $f$ be a linearly non-degenerate meromorphic mapping of $M$ into $\P^n(\C)\ (n\ge 2)$. Let $H_1,\ldots,H_q$ be $q$ hyperplanes of $\P^n(\C)$ in general possition. Assume that $f$ satisfies the condition $(C_\rho)$ and
$$\dim f^{-1}(H_i)\cap f^{-1}(H_j) \le m-2 \quad (1 \le i<j \le q).$$ 
If $q>2n+1+\dfrac{2n}{3n+1}+\rho\dfrac{(n^2+4q-3n)(6n+1)}{6n^2+2}$ then $\sharp \mathcal {F}(f,\{H_i\}_{i=1}^q,1)\le 2$.
\end{theorem}
Here, by $\sharp S$ we denote the cardinality of the set $S$.

Note: If $M=\C^m$ then $\rho=0$, and hence Theorem \ref{1.1} immediately implies Theorem A.

In the last section of this paper, we will extend the uniqueness theorems for meromorphic mappings of $\C^m$ into $\P^n(\C)$ sharing $2n+2$ hyperplanes (see  \cite{HQ19,SDQ11}) to the case of K\"{a}hler manifolds. Our last result is stated as follows.
\begin{theorem}\label{1.2}
Let $M$ be an $m$-dimensional connected K\"{a}hler manifold whose universal covering is biholomorphic to is biholomorphic to a ball $\B^m(R_0)$ in $\C^m$\ ($0<R_0\le +\infty$), and let $f$ be a linearly non-degenerate meromorphic mapping of $M$ into $\P^n(\C)\ (n\ge 2)$ satisfying the condition $(C_\rho)$. Let $H_1,\ldots,H_q$ be $q$ hyperplanes of $\P^n(\C)$ in general possition such that
$$\dim f^{-1}(H_i)\cap f^{-1}(H_j) \le m-2 \quad (1 \le i<j \le q).$$ 
Assume that 
$$q>2n+1+\frac{6np}{6np+1}+\rho \left (2n(n+1)+\frac{6n^2(n+1)p+np(p-2)}{6np+1}\right),$$
where $p=\binom{2n+2}{n+1}$. Then $\sharp \mathcal {F}(f,\{H_i\}_{i=1}^q,n+1)=1$.
\end{theorem}
If $M=\C^m$ then $\rho=0$ and the above theorem implies the following corollary, which is a weak form of \cite[Theorem 1.2]{SDQ11}.
\begin{corollary}
Let $f$ be a linearly non-degenerate meromorphic mapping of $\C^m$ into $\P^n(\C)\ (n\ge 2)$. Let $H_1,\ldots,H_{2n+2}$ be $2n+2$ hyperplanes of $\P^n(\C)$ in general possition. Assume that 
$$\dim f^{-1}(H_i)\cap f^{-1}(H_j) \le m-2 \quad (1 \le i<j \le 2n+2).$$ 
Then $\sharp \mathcal {F}(f,\{H_i\}_{i=1}^{2n+2},n+1)=1$.
\end{corollary}
\section{Basic notions and auxiliary results from Nevanlinna theory}

\noindent
{\bf 2.1. Counting function.}\ We set $\|z\| = \big(|z_1|^2 + \dots + |z_m|^2\big)^{1/2}$ for
$z = (z_1,\dots,z_m) \in \mathbb C^m$ and define
\begin{align*}
\B^m(R) &:= \{ z \in \mathbb C^m : \|z\| < R\}\ \   (0<R\le \infty),\\
S(R) &:= \{ z \in \mathbb C^m : \|z\| =R\}\ (0<R<\infty).
\end{align*}
Define 
$$v_{m-1}(z) := \big(dd^c \|z\|^2\big)^{m-1}\quad \quad \text{and}$$
$$\sigma_m(z):= d^c \text{log}\|z\|^2 \land \big(dd^c \text{log}\|z\|^2\big)^{m-1}
 \text{on} \quad \mathbb C^m \setminus \{0\}.$$

 For a divisor $\nu$ on  a ball $\B^m(R)$ of $\mathbb C^m$, and for a positive integer $p$ or $p= \infty$, we define the counting function of $\nu$ by
\begin{align*}
&\nu^{[p]}(z)=\min\ \{p,\nu(z)\},\\
&n(t) =
\begin{cases}
\int\limits_{|\nu|\,\cap \B(t)}
\nu(z) v_{m-1} & \text  { if } m \geq 2,\\
\sum\limits_{|z|\leq t} \nu (z) & \text { if }  m=1. 
\end{cases}
\end{align*}
Similarly, we define $n^{[p]}(t).$

Define
$$ N(r,r_0,\nu)=\int\limits_{r_0}^r \dfrac {n(t)}{t^{2m-1}}dt \quad (0<r_0<r<R).$$

Similarly, define $N(r,r_0,\nu^{[p]})$ and denote it by $N^{[p]}(r,r_0,\nu)$.

Let $\varphi : \mathbb B^m(R) \longrightarrow \overline\C $ be a meromorphic function. Denote by $\nu_\varphi$ (res. $\nu^{0}_{\varphi}$) the divisor (resp. the zero divisor) of $\varphi$. Define
$$N_{\varphi}(r,r_0)=N(r,r_0,\nu^{0}_{\varphi}), \ N_{\varphi}^{[p]}(r,r_0)=N^{[p]}(r,r_0,(\nu^{0}_{\varphi})^{[p]}).$$
For brevity, we will omit the character $^{[p]}$ if $M=\infty$.

\vskip0.2cm
\noindent
{\bf 2.2. Characteristic function.}\ Throughout this paper, we fix a homogeneous coordinates system $(x_0:\cdots :x_n)$ on $\P^n(\C)$. Let $f : \mathbb B^m(R) \longrightarrow \mathbb P^n(\mathbb C)$ be a meromorphic mapping with a reduced representation $f = (f_0, \ldots,f_n)$, which means that each $f_i$ is a holomorphic function on $\B^m(R)$ and  $f(z) = \big(f_0(z) : \dots : f_n(z)\big)$ outside the indeterminancy locus $I(f)$ of $f$. Set $\|f \| = \big(|f_0|^2 + \dots + |f_n|^2\big)^{1/2}$.

The characteristic function of $f$ is defined by 
$$ T_f(r,r_0)=\int_{r_0}^r\dfrac{dt}{t^{2m-1}}\int\limits_{B(t)}f^*\Omega\wedge v^{m-1}, \ (0<r_0<r<R). $$
By Jensen's formula, we have
\begin{align*}
T_f(r,r_0)= \int\limits_{S(r)} \log\|f\| \sigma_m -
\int\limits_{S(r_0)}\log\|f\|\sigma_m +O(1), \text{ (as $r\rightarrow R$)}.
\end{align*}

\vskip0.2cm
\noindent
{\bf 2.3. Auxiliary results.}\ Repeating the argument in \cite[Proposition 4.5]{Fu85}, we have the following.

\begin{proposition}\label{2.1}
Let $F_0,\ldots ,F_{l-1}$ be meromorphic functions on the ball $\B^{m}(R_0)$ in $\mathbb C^m$ such that $\{F_0,\ldots ,F_{l-1}\}$ are  linearly independent over $\mathbb C.$ Then  there exists an admissible set  
$$\{\alpha_i=(\alpha_{i1},\ldots,\alpha_{im})\}_{i=0}^{l-1} \subset \mathbb N^m,$$
which is chosen uniquely in an explicit way, with $|\alpha_i|=\sum_{j=1}^{m}|\alpha_{ij}|\le i \ (0\le i \le l-1)$ such that:

(i)\  $W_{\alpha_0,\ldots ,\alpha_{l-1}}(F_0,\ldots ,F_{l-1})\overset{Def}{:=}\det{({\mathcal D}^{\alpha_i}\ F_j)_{0\le i,j\le l-1}}\not\equiv 0.$ 

(ii) $W_{\alpha_0,\ldots ,\alpha_{l-1}}(hF_0,\ldots ,hF_{l-1})=h^{l+1}W_{\alpha_0,\ldots ,\alpha_{l-1}}(F_0,\ldots ,F_{l-1})$ for any nonzero meromorphic function $h$ on $\B^m(R_0).$
\end{proposition}
The function $W_{\alpha_0,\ldots ,\alpha_{l-1}}(F_0,\ldots ,F_{l-1})$ is called the general Wronskian of the mapping $F=(F_0,\ldots ,F_{l-1})$.

\begin{definition}[Cartan's auxialiary function]\label{2.2}
For meromorphic functions $F,G,H$ on $\B^m(R_0)$ and $\alpha =(\alpha_1,\ldots ,\alpha_m)\in \Z_+^m$, we define the Cartan's auxiliary function as follows:
$$
\Phi^\alpha(F,G,H):=F\cdot G\cdot H\cdot\left | 
\begin {array}{cccc}
1&1&1\\
\frac {1}{F}&\frac {1}{G} &\frac {1}{H}\\
\mathcal {D}^{\alpha}(\frac {1}{F}) &\mathcal {D}^{\alpha}(\frac {1}{G}) &\mathcal {D}^{\alpha}(\frac {1}{H})
\end {array}
\right|.
$$
\end{definition}

\begin{lemma}[{see \cite[Proposition 3.4]{Fu98}}]\label{2.3} If $\Phi^\alpha(F,G,H)=0$ and $\Phi^\alpha(\frac {1}{F},\frac {1}{G},\frac {1}{H})=0$ for all $\alpha$ with $|\alpha|\le 1$, then one of the following assertions holds:

(i) \ $F=G, G=H$ or $H=F$,

(ii) \ $\frac {F}{G},\frac {G}{H}$ and $\frac {H}{F}$ are all constant.
\end{lemma}

\begin{lemma}[{see \cite[Lemma 2.2]{SDQ19}}]\label{2.4}
Let $f$ be a meromorphic mapping from $\B^{m}(R_0)\ (0<R_0\le +\infty)$ into $\P^n(\C)$. Let $f^1,f^2,f^3$ be three maps in $\mathcal F(f,\{H_i\}_{i=1}^q,1)$. Assume that each $f^i$ has a representation $f^i=(f^i_{0}:\cdots :f^i_{n})$, $1\le i\le 3$. Suppose that there exist $s,t,l\in\{1,\ldots ,q\}$ such that
$$ 
P:=\mathrm{det}\left (\begin{array}{ccc}
(f^1,H_s)&(f^1,H_t)&(f^1,H_l)\\ 
(f^2,H_s)&(f^2,H_t)&(f^2,H_l)\\
(f^3,H_s)&(f^3,H_t)&(f^3,H_l)
\end{array}\right )\not\equiv 0.
$$
Then we have
\begin{align*}
\nu_P\ge\sum_{i=s,t,l}\bigl (\min_{1\le u\le 3}\{\nu_{(f^u,H_i)}\}-\nu^{[1]}_{(f,H_i)}\bigl )
+ 2\sum_{i=1}^q\nu^{[1]}_{(f,H_i)}.
\end{align*}
\end{lemma}

Let $G$ be a  torsion free abelian group and let $A=(a_1,a_2,\ldots,a_q)$ be a $q-$tuple of elements $a_i$ in $G$. Let $q\ge r>s>1.$ We say that the $q-$tuple $A$ has the property $(P_{r,s})$ if any $r$ elements $a_{l(1)},\ldots,a_{l(r)}$ in $A$ satisfy the condition that for any given $i_1,\ldots,i_s \ (1\le i_1<\cdots<i_s\le r)$, there exist $j_1,\ldots,j_s \ (1\le j_1<\cdots<j_s\le r)$ with $\{i_1,\ldots,i_s\}\neq\{j_1,\ldots,j_s\}$ such that $a_{l(i_1)}\cdots a_{l(i_s)}=a_{l(j_1)}\cdots a_{l(j_s)}.$

\begin{proposition}[{See H. Fujimoto \cite{Fu75}}]\label{2.5} 
Let $G$ be a  torsion free abelian group and $A=(a_1,\ldots,a_q)$ be a $q-$tuple of elements $a_i$ in $G$. If $A$ has the property $(P_{r,s})$ for some $r,s$ with $q\ge r>s>1,$ then there exist $i_1,\ldots,i_{q-r+2}$ with $1\le i_1<\cdots<i_{q-r+2}\le q$ such that 
$a_{i_1}=a_{i_2}=\cdots=a_{i_{q-r+2}}.$
\end{proposition}

\section{Functions of small integration}

Let $f^1,f^2,\ldots,f^k$ be $m$ meromorphic mappings from the complete K\"{a}hler manifold $\B^m(1)$ into $\P^n(\C)$, which satisfies the condition $(C_\rho)$ for a non-negative number $\rho$. 
For each $1\le u\le k$, we fix a reduced representation $f^u=(f^u_0:\cdots :f^u_n)$ of $f^u$ and set $\|f^u\|=(|f^u_0|^2+\cdots+|f^u_n|^2)^{1/2}$.

We denote by $\mathcal C(\B^m(R_0))$ the set of all non-negative functions $g: \B^m(R_0)\to [0,+\infty]$ which are continuous outside an analytic set of codimension two (corresponding to the topology of the compactification $[0,+\infty]$) and only attain $+\infty$ in an analyic thin set.

\begin{definition}[Functions of small integration]\label{3.1}
A function $g$ in $\mathcal C(\B^m(R_0))$ is said to be of small integration with respective to $f^1,\ldots,f^k$ at level $l_0$ if there exist an element $\alpha=(\alpha_1,\ldots,\alpha_m)\in\N^m$ with $|\alpha|\le l_0$,  a positive number $K$, such that for every $0\le tl_0<p<1$,
$$\int_{S(r)}|z^\alpha g|^t\sigma_m \le K\left(\frac{R^{2m-1}}{R-r}\sum_{u=1}^kT_{f^u}(r,r_0)\right)^p$$
for all $r$ with $0<r_0<r<R<R_0$, where $z^\alpha=z_1^{\alpha_1}\cdots z_m^{\alpha_m}$. 
\end{definition}


We denote by $S(l_0;f^1,\ldots,f^k)$ the set of all functions in $\mathcal C(\B^m(R_0))$ which are of small integration with respective to $f^1,\ldots,f^k$ at level $l_0$. We see that, if $g$ belongs to $S(l_0;f^1,\ldots,f^k)$ then $g$ is also belongs to $S(l;f^1,\ldots,f^k)$ for every $l>l_0$. Moreover, if $g$ is a constant function then $g\in S(0;f^1,\ldots,f^k)$.

\begin{proposition}\label{3.2}
If $g_i\in S(l_i;f^1,\ldots,f^l)\ (1\le i\le s)$ then $\prod_{i=1}^sg_i\in S(\sum_{i=1}^sl_i;f^1,\ldots,f^l)$.
\end{proposition}
\begin{proof}
We take the element $\alpha^i=(\alpha^i_1,\ldots,\alpha^i_m)$ with respect to $g_i$ as in the above definition. Then, for every $1\le t\sum_{i=1}^sl_i<p<1$, by H\"{o}lder inequality we have
\begin{align*}
 \int_{S(r)}|z^{\alpha^1+\cdots +\alpha^s}g_1\cdots g_s|^t\sigma_m&\le \prod_{i=1}^s\left (\int_{S(r)}|z^{\alpha^i}g_i|^{(t\sum_{j=1}^sl_j)/l_i}\sigma_m\right)^{l_i/\sum_{j=1}^sl_j}\\
&\le \left (K\left(\frac{R^{2m-1}}{R-r}\sum_{u=1}^kT_{f^u}(r,r_0)\right)^p\right)^{\sum_{i=1}^sl_i/\sum_{j=1}^sl_j}\\
&=K\left(\frac{R^{2m-1}}{R-r}\sum_{u=1}^kT_{f^u}(r,r_0)\right)^p,
\end{align*}
for every $r$, $0<r_0<r<R<1$.
Therefore, $g_1\cdots g_s\in R(\sum_{i=1}^sl_i;f^1,\ldots,f^k)$.
\end{proof}

\begin{definition}[Functions of bounded integration]\label{3.3} 
A meromorphic function $h$ on $\B^m(R_0)$ is said to be of bounded integration with bi-degree $(p,l_0)$ for the family $\{f^1,\ldots,f^k\}$ if there exists $g\in S(l_0;f^1,\ldots,f^k)$ satisfying
$$|h|\le \|f^1\|^p\cdots \|f^u\|^p\cdot g,$$
outside a proper analytic subset of $\B^m(R_0)$.
\end{definition}
Denote by $B(p,l_0;f^1,\ldots,f^k)$ the set of all meromorphic functions on $\B^m(R_0)$ which are of bounded integration of bi-degree $(p,l_0)$ for $\{f^1,\ldots,f^k\}$. We have the following:
\begin{itemize}
\item For a meromorphic mapping $h$, $|h|\in S(l_0;f^1,\ldots,f^k)$ iff $h\in B(0,l_0;f^1,\ldots,f^k)$.
\item $B(p,l_0;f^1,\ldots,f^k)\subset B(p,l;f^1,\ldots,f^k)$ for every $0\le l_0<l$.
\item If $h_i\in B(p_i,l_i;f^1,\ldots,f^k)\ (1\le i\le s)$ then 
$$h_1\cdots h_m\in B(\sum_{i=1}^sp_i,\sum_{i=1}^sl_i;f^1,\ldots,f^k).$$
\end{itemize}

The following proposition is proved by Fujimoto \cite{Fu83} and reproved by Ru-Sogome \cite{RS}.
\begin{proposition}[{see \cite[Proposition 6.1]{Fu83}, also \cite[Proposition 3.3]{RS}}]\label{3.4}
 Let $L_1,\ldots ,L_l$ be linear forms of $l$ variables and assume that they are linearly independent. Let $F$ be a meromorphic mapping from the ball $\B^m(R_0)\subset\C^m$ into $\P^{l-1}(\C)$ with a reduced representation $F=(F_0,\ldots ,F_{l-1})$ and let $(\alpha_1,\ldots ,\alpha_l)$ be an admissible set of $F$. Set $l_0=|\alpha_1|+\cdots +|\alpha_l|$ and take $t,p$ with $0< tl_0< p<1$. Then, for $0 < r_0 < R_0,$ there exists a positive constant $K$ such that for $r_0 < r < R < R_0$,
$$\int_{S(r)}\biggl |z^{\alpha_1+\cdots +\alpha_l}\dfrac{W_{\alpha_1,\ldots ,\alpha_l}(F_0,\ldots ,F_{l-1})}{L_0(F)\ldots L_{l-1}(F)}\biggl |^t\sigma_m\le K\biggl (\dfrac{R^{2m-1}}{R-r}T_F(R,r_0)\biggl )^p.$$
\end{proposition}
This proposition implies that the function $\left |\dfrac{W_{\alpha_1,\ldots ,\alpha_l}(F_0,\ldots ,F_{l-1})}{L_0(F)\ldots L_{l-1}(F)}\right |$ belongs to $S(l_0;F)$.

We will prove the following proposition, which can be considered as a general form of finiteness theorems for meromorphic mappings on K\"{a}hler manifold into projective space.
\begin{proposition}\label{3.5} 
	Let $M$ be a complete connected K\"{a}hler manifold whose universal covering is biholomorphic to a ball $\B^m(R_0)\ (0<R_0\le +\infty)$. Let $f^1,f^2,\ldots,f^k$ be $m$ linearly non-degenerate meromorphic mappings from $M$ into $\P^n(\C)$, which satisfy the condition $(C_\rho)$. Let $H_1,\ldots,H_q$ be $q$ hyperplanes of $\P^n(\C)$ in general position, where $q$ is a positive integer. Assume that there exists a non zero holomorphic function $h\in B(p,l_0;f^1,\ldots,f^k)$ such that
	$$\nu_h\ge\lambda\sum_{u=1}^k\sum_{i=1}^q\nu^{[n]}_{(f^u,H_i)},$$
	where $p,l_0$ are non-negative integers, $\lambda$ is a positive number. Then we have
	$$q\le n+1+\rho k\frac{n(n+1)}{2}+\frac{1}{\lambda}\left (p+\rho l_0\right).$$
\end{proposition}
\begin{proof} Without loss of generality, we may assume that $M=\B^m(R_0)$. 

If $R_0=+\infty$, by the second main theorem we have
\begin{align*}
(q-n-1)\sum_{u=1}^kT_{f^u}(r,1)&\le\sum_{u=1}^k\sum_{i=1}^qN^{[n]}_{(f^u,H_i)}(r,1)+o(\sum_{u=1}^kT_{f^u}(r,1))\\
&\le \frac{1}{\lambda}N_h(r,1)+o(\sum_{u=1}^kT_{f^u}(r,1))\\
&=\frac{p}{\lambda}\sum_{u=1}^kT_{f^u}(r,1)+o(\sum_{u=1}^kT_{f^u}(r,1)),
\end{align*}
for all $r\in [1;+\infty)$ outside a Lebesgue set of finite measure. 
Letting $r\rightarrow +\infty$, we obtain
$$ q\le n+1+\frac{p}{\lambda}.$$

Now, we consider the case where $R_0<+\infty$. Without loss of generality we assume that $R_0=1$. Suppose contrarily that $q>n+1+\rho k\dfrac{n(n+1)}{2}+\dfrac{1}{\lambda}\left (p+\rho l_0\right)$. Then, there is a positive constant $\epsilon$ such that
	$$q> n+1+\rho k\frac{n(n+1)}{2}+\frac{1}{\lambda}\left (p+\rho (l_0+\epsilon)\right).$$ 
	Put $l_0'=l_0+\epsilon >0$.

	Suppose that $f^u$ has a reduced representation $f^u=(f_{0}^u:\cdots :f_{n}^u)$ for each $1\le u\le k$. Since $f^u$ is linearly non-degenerate, there exists an admissible set $(\alpha_0^u,\ldots,\alpha_n^u)\in\mathbb (\N^m)^{n+1}$ with $|\alpha_i^u|\le i\ (0\le i\le n)$ such that the general Wronskian
	$$ W(f^u):=\det\left (\mathcal D^{\alpha_i^u}(f_j^u); 0\le i,j\le n\right)\not\equiv 0. $$
	By usual argument of Nevanlinna theory, we have
	$$\nu_h\ge \lambda\sum_{u=1}^k\sum_{i=1}^q\nu^{[n]}_{(f^u,H_i)}\ge\lambda\sum_{u=1}^k\left (\sum_{i=1}^q\nu_{(f^u,H_i)}-\nu_{W(f^u)}\right).$$
	Put $w_u(z):=z^{\alpha_0^u+\cdots +\alpha_n^u}\dfrac{W(f^u)}{\prod_{i=1}^{q}(f,H_i)}\ (1\le u\le 3)$. Since $h\in B(p,l_0;f^1,\ldots,f^k)$, there exists a non-negative plurisubharmonic function $g\in S(l_0;f^1,\ldots,f^k)$ and $\beta =(\beta_1,\ldots,\beta_m)\in\mathbb Z^{m}_+$ with $|\beta|\le l_0$ such that 
	\begin{align}\label{3.6}
	\int_{S(r)}\left |z^\beta g\right|^{t'}\sigma_m=O\left (\frac{R^{2m-1}}{R-r}\sum_{u=1}^kT_{f^u}(r,r_0)\right )^l,
	\end{align}
for every $0\le l_0t'<l<1$ and 
\begin{align}\label{3.7}
|h|\le \left (\prod_{u=1}^k\|f^u\|\right )^p |g|.
\end{align}

	Put $ t=\frac{\rho}{q-n-\frac{p}{\lambda}-1}>0$ and $ \phi:=|w_1|\cdots |w_k|\cdot|z^\beta h|^{1/\lambda}.$
	Then $a=t\log\phi$ is a plurisubharmonic function on $\B^m(1)$ and
	$$\left (k\frac{n(n+1)}{2}+\frac{l_0'}{\lambda}\right)t< 1.$$
	Therefore, we may choose a positive number $p'$ such that $0\le (k\frac{n(n+1)}{2}+\frac{l_0'}{\lambda})t<p'<1.$
	
	Since $f^u$ satisfies the condition $(C_\rho)$, then there exists a continuous plurisubharmonic function $\varphi_u$ on $\B^m(1)$ such that
	$$ e^{\varphi_u}dV \le \|f^u\|^{\rho}v_m.$$
	We see that $\varphi=\varphi_1+\cdots+\varphi_k+a$ is a plurisubharmonic function on $\B^m(1)$. We have
	\begin{align*}
	e^\varphi dV&=e^{\varphi_1+\cdots+\varphi_k+t\log\phi}dV\le e^{t\log\phi}\prod_{u=1}^k\|f^u\|^{\rho}v_m=|\phi|^{t}\prod_{u=1}^k\|f^u\|^{\rho}v_m\\
	&=|z^\beta g|^{t/\lambda}\prod_{u=1}^k(|w_u|^t\cdot\|f^u\|^{\rho+pt/\lambda})v_m=|z^\beta g|^{t/\lambda}\prod_{u=1}^k(|w_u|^t\cdot\|f^u\|^{(q-n-1)t})v_m.
	\end{align*}
	Setting $x=\dfrac{l_0'/\lambda}{kn(n+1)/2+l_0'/\lambda}$, $y=\dfrac{n(n+1)/2}{kn(n+1)/2+l_0'/\lambda}$, then we have $x+ky=1$. Therefore, by integrating both sides of the above inequality over $\B^m(1)$ and applying H\"{o}lder inequality,  we have
	\begin{align}\label{3.8}
	\begin{split}
	\int_{\B^m(1)}e^udV&\le \int_{\B^m(1)}\prod_{u=1}^k(|w_u|^t\cdot\|f^u\|^{(q-n-1)t})|z^\beta g|^{t/\lambda}v_m\\
	&\le \left (\int_{\B^m(1)}|z^\beta g|^{t/(\lambda x)}v_m\right)^{x}\\
	&\ \ \times\prod_{u=1}^k\left (\int_{\B^m(1)}(|w_u|^{t/y}\cdot\|f^u\|^{(q-n-1)t/y})v_m\right)^{y}\\
	&\le \left (2m\int_0^1r^{2m-1}\left (\int_{S(r)}|z^\beta g|^{t/(\lambda x)}\sigma_m\right )dr\right)^{x}\\
	&\ \ \times\prod_{u=1}^k\left (2m\int_0^1r^{2m-1}\left (\int_{S(r)}\bigl (|w_u|\cdot\|f^u\|^{(q-n-1)}\bigl )^{t/y}\sigma_m\right )dr\right)^{y}.
	\end{split}
	\end{align}
	
	(a) We now deal with the case where
	$$ \lim\limits_{r\rightarrow 1}\sup\dfrac{\sum_{u=1}^kT_{f^u}(r,r_0)}{\log 1/(1-r)}<\infty .$$
	We see that $\dfrac{l_0t}{\lambda x}\le \dfrac{l_0't}{\lambda x}=\bigl (k\dfrac{n(n+1)}{2}+\dfrac{l_0'}{\lambda}\bigl )t<p'$ and $\dfrac{n(n+1)}{2}\dfrac{t}{y}=\bigl (k\dfrac{n(n+1)}{2}+\dfrac{l_0'}{\lambda}\bigl )t<p'$.  By \cite[Proposition 6.1]{Fu85} and (\ref{3.6}), there exists a positive constant $K$ such that, for every $0<r_0<r<r'< 1,$ we have
	\begin{align*}
	&\int_{S(r)}\bigl (|w_u|\cdot\|f^u\|^{(q-n-1)}\bigl )^{t/y}\sigma_m\le K\left (\dfrac{r'^{2m-1}}{r'-r}T_{f^u}(r',r_0)\right )^{p'}\ (1\le u\le k)\\
	\text{and }&\int_{S(r)}|z^\beta g|^{t/(\lambda x)}\sigma_m\le K\left (\dfrac{r'^{2m-1}}{r'-r}\sum_{u=1}^kT_{f^u}(r',r_0)\right )^{p'}.
	\end{align*}
		
	Choosing $r'=r+\dfrac{1-r}{e\max_{1\le u\le k}T_{f^u}(r,r_0)}$, we have
	$ T_{f^u}(r',r_0)\le 2T_{f^u}(r,r_0)$,
	for all $r$ outside a subset $E$ of $(0,1]$ with $\int_E\frac{1}{1-r}dr<+\infty$.
	Hence, the above inequality implies that
	\begin{align*}
	&\int_{S(r)}\bigl (|w_u|\cdot\|f^u\|^{(q-n-1)}\bigl )^{t/y}\sigma_m\le \dfrac{K'}{(1-r)^{p'}}\left (\log\frac{1}{1-r}\right )^{2p'}\ (1\le u\le k)\\
	\text{and }&\int_{S(r)}|z^\beta g|^{t/(\lambda x)}\sigma_m\le \dfrac{K'}{(1-r)^{p'}}\left (\log\frac{1}{1-r}\right )^{2p'}
	\end{align*}
	for all $r$ outside $E$, and for some positive constant $K'$. Then the inequality (\ref{3.8}) yields that
	\begin{align*}
	\int_{\B^m(1)}e^udV&\le 2m\int_0^1r^{2m-1}\frac{K'}{1-r}\left (\log\frac{1}{1-r}\right)^{2p'}dr< +\infty.
	\end{align*}
	This contradicts the results of S.T. Yau \cite{Y76} and L. Karp \cite{K82}. 
	
	(b) We now deal with the remaining case where 
	$$ \lim\limits_{r\rightarrow 1} \sup\dfrac{\sum_{u=1}^kT_{f^u}(r,r_0)}{\log 1/(1-r)}= \infty.$$
	As above, we have
	$$\int_{S(r)}|z^\beta g|^{t/(\lambda x)}\sigma_m\le K\left (\dfrac{1}{1-r}\sum_{u=1}^kT_{f^u}(r,r_0)\right )^{p'}$$
	for every $r_0<r<1.$
	By the concativity of logarithmic function, we have
	$$ \int_{S(r)}\log|z^\beta|^{t/(\lambda x)}\sigma_m+\int_{S(r)}\log|g|^{t/(\lambda x)}\sigma_m\le K''\left (\log^+\frac{1}{1-r}+\log^+\sum_{u=1}^kT_{f^u}(r,r_0)\right). $$
	This implies that
	$$\int_{S(r)}\log|g|\sigma_m= O\left (\log^+\frac{1}{1-r}+\log^+\sum_{u=1}^kT_{f^u}(r,r_0)\right) $$
	By \cite[proposition 6.2]{Fu85} and (\ref{3.7}), we have
	\begin{align*}
	\sum_{u=1}^kpT_{f^u}(r,r_0)+\int_{S(r)}\log|g|\sigma_m&\ge N_h(r,r_0)+S(r)\ge\lambda\sum_{u=1}^k\sum_{i=1}^qN_{(f,H_i)}^{[n]}(r,r_0)+S(r)\\ 
	&\ge \lambda\sum_{u=1}^k(q-n-1)T_{f^u}(r,r_0)+S(r), 
	\end{align*}
	where $S(r)=O(\log^+\frac{1}{1-r}+\log^+\sum_{u=1}^kT_{f^u}(r_0,r))$ for every $r$ excluding a set $E$ with $\int_E\frac{dr}{1-r}<+\infty$.
	Letting $r\rightarrow 1$, we get $\frac{p}{\lambda}>q-n-1$. This is a contradiction.
	Hence, the supposition is false The proposition is proved.
\end{proof}

\section{Proof of  Theorem \ref{1.1}} Since the case where $M=\C^m$ have already proved by the author in \cite{SDQ19}, without loss of generality, in this proof we only consider the case where $M=\mathbb B^m(1)$. 

Hence, $f$ is a linearly non-degenerate meromorphic mapping of $\B^m(1)$ into $\P^n(\C )$ and $H_1,\ldots ,H_{2n+2}$ be $2n+2$ hyperplanes of $\P^n(\C )$ in general position with
$$\dim f^{-1}(H_i)\cap f^{-1}(H_j) \le m-2 \quad (1 \le i<j \le 2n+2).$$ 

In order to prove Theorem \ref{1.1}, we need the following lemmas.
\begin{lemma}\label{4.1}
 If $q>2n+1+\dfrac{2n}{3n+1}+\rho\dfrac{(n^2+4q-3n)(6n+1)}{6n^2+2}$ then the following hold:
\begin{item}
\item[\indent (i)] $q>n+1+3\rho\dfrac{n(n+1)}{2}+\dfrac{3nq}{2q+2n-2}$.
\item[\indent (ii)] $q> n+1+3\rho\dfrac{n(n+1)+4(q-n)}{2}$.
\item[\indent (iii)] $q> n+1+3\rho\dfrac{n(n+2)}{2}+\dfrac{3n}{11n-6}((q-1)+\rho(q-1)).$
\item[\indent (iv)] $q> n+1+3\rho\dfrac{n(n+1)}{2}+\dfrac{3n(q^2+\rho q(q-2))}{6nq+(n-2)(q-2)+4q-6n-2}$.
\end{item}
\end{lemma}
\begin{proof}
 From the assumption, we have 
$$q>n+1+\dfrac{3nq}{6n+1}+3\rho\frac{n(n+1)+4(q-n)}{2},$$ 
and also $q\ge 2n+2.$ Then we have 
$$ \frac{3nq}{2q+2n-2}\le \frac{3nq}{6n+2}<\frac{3nq}{6n+1}.$$
This implies the inequality (i). 

The inequality (ii) and (iii) are clear. We now show that the inequality (iv) is also satisfied. Indeed, we have
$$\frac{3nq^2}{6nq+(n-2)(q-2)+4q-6n-2}< \frac{3nq^2}{6nq+q}=\frac{3nq}{6n+1}$$
and 
\begin{align*}
\frac{\rho q(q-2)}{6nq+(n-2)(q-2)+4q-6n-2}\le \frac{\rho q(q-2)}{6nq+q}=\rho\frac{q-2}{6n+1}\le 6\rho(q-n).
\end{align*}
This implies the inequality (iv).
\end{proof}

\begin{lemma}\label{4.2} 
Let $f$ be a linearly non-degenerate meromorphic mapping from $\mathbb B^m(1)$ into $\P^n(\C)$, which satisfies the condition $(C_\rho)$. Let $q$ be a positive integer with $q\ge 2n+1+\rho n(n+1)$. Then every mapping $g\in\mathcal F(f,\{H_i\}_{i=1}^{q},1)$ is linearly non-degenerate.
\end{lemma}
\begin{proof}
Suppose contrarily that there exists a hyperplane $H$ satisfying $g(\C^m)\subset H$. We assume that $f$ and $g$ have reduced representations $f=(f_{0}:\cdots :f_{n})$ and $g=(g_0:\cdots :g_n)$ respectively. Since $f$ is linearly non-degenerate, there exists an admissible set $(\alpha_0,\ldots,\alpha_n)\in(\N^m)^{n+1}$ with $|\alpha_i|\le i\ (0\le i\le n)$ such that
$$ W(f):=\det\left (\mathcal D^{\alpha_i}(f_j); 0\le i,j\le n+1\right)\not\equiv 0. $$
Assume that $H=\{(\omega_0:\cdots :\omega_n)\ |\ \sum_{i=0}^na_i\omega_i=0\}$. 
Since $f$ is linearly non-degenerate, $(f,H)\not\equiv 0$. On the other hand $(f,H)(z)=(g,H)(z)=0$ for all $z\in\bigcup_{i=1}^{q}f^{-1}(H_i)$, hence
\begin{align*}
\nu_{(f,H)^n}&\ge n\sum_{i=1}^{q}\nu_{(f,H_i)}^{[1]}\ge\sum_{i=1}^{q}\nu_{(f,H_i)}^{[n]}.
\end{align*}
We see that $|(f,H)^n|\le C\cdot \|f\|^n$ for a positive constant $C$. Then $(f,H)^n\in B(n,0;f)$. Hence, applying Proposition \ref{3.5} for the function $(f,H)^n$ and $q$ hyperplanes $H_1,\ldots,H_q$, we deduce that
$$ q\le n+1+\rho n(n+1)+n=2n+1+\rho n(n+1).$$
This is a contradiction.
\end{proof}

\vskip0.2cm
Now for three mappings $f^1, f^2, f^3 \in \mathcal {F}(f,\{H_i\}_{i=1}^{q},1)$, we define:

$\bullet$ $F_k^{ij}=\dfrac {(f^k,H_i)}{(f^k,H_j)}\ (0\le k\le 2,\ 1\le i,j\le 2n+2),$

$\bullet$ $V_i=( (f^1,H_i),(f^2,H_i),(f^3,H_i))\in\mathcal M_m^3,$

$\bullet$ $\nu_i$: the divisor whose support is the closure of the set $\{z;\nu_{(f^u,H_i)}(z)\ge \nu_{(f^v,H_i)}(z)=\nu_{(f^t,H_i)}(z)\text{ for a permutation } (u,v,t) \text{ of } (1,2,3)\}$.

We write $V_i\cong V_j$ if $V_i\wedge V_j\equiv 0$, otherwise we write $V_i\not\cong V_j.$ For $V_i\not\cong V_j$, we write $V_i\sim V_j$ if there exist $1\le u<v\le 3$ such that $F_u^{ij}=F_v^{ij}$, otherwise we write $V_i\not\sim V_j$.

\begin{lemma}\label{4.3}
If $n\ge 2$ then $f^1\wedge f^2\wedge f^3=0$ for every $f^1,f^2,f^3\in\mathcal{F}(f,\{H_i\}_{i=1}^{q},1)$ provided
$$ q>n+1+3\rho\frac{n(n+1)}{2}+\frac{3nq}{2q+2n-2}.$$
\end{lemma}
This lemma is firstly proved in \cite[Theorem 1]{NT}. For the sake of completeness, we will give another short simple proof.
\begin{proof}
Suppose that $f^1\wedge f^2\wedge f^3\not\equiv 0$. Without loss of generality, we may assume that
$$\underbrace{V_1\cong \cdots\cong V_{l_1}}_{\text { group } 1}\not\cong
\underbrace{V_{l_1+1}\cong\cdots\cong V_{l_2}}_{\text { group } 2}\not\equiv \underbrace{V_{l_2+1}\cong \cdots\cong V_{l_3}}_{\text { group } 3}\not\cong\cdots \not\cong\underbrace{V_{l_{s-1}}\cong\cdots\cong V_{l_{s}}}_{\text { group } s},$$
where $l_s=q.$ 
 For each $1\le i\le q$, we set
$$ \sigma (i)=\begin{cases}
i+n&\text{ if }i+n\le q,\\
i+n-q&\text{ if }i+n>q.
\end{cases} $$
Then, we see that $V_i$ and $V_{\sigma (i)}$ belong to two distinct groups, i.e., $V_i\wedge V_{\sigma (i)}\not\equiv 0$. Therefore, we may choose another index, denoted by $\gamma (i)$, such that 
$$V_i\wedge V_{\sigma (i)}\wedge V_{\gamma (i)}\not\equiv 0.$$
We set
$$ P_i:= \mathrm{det}\left (\begin{array}{ccc}
(f^1,H_i)&(f^1,H_{\sigma (i)})&(f^1,H_{\gamma (i)})\\ 
(f^2,H_i)&(f^2,H_{\sigma (i)})&(f^2,H_{\gamma (i)})\\
(f^3,H_i)&(f^3,H_{\sigma (i)})&(f^3,H_{\gamma (i)})
\end{array}\right )\not\equiv 0.$$
Then, by Lemma \ref{2.4} we have
\begin{align*}
\nu_{P_i}&\ge \bigl (\min_{1\le u\le 3}\{\nu_{(f^u,H_i)}\}-\nu^{[1]}_{(f,H_i)}\bigl )+\bigl (\min_{1\le u\le 3}\{\nu_{(f^u,H_{\sigma (i)})}\}-\nu^{[1]}_{(f,H_{\sigma (i)})}\bigl )+2\sum_{j=1}^q\nu^{[1]}_{(f,H_j)}\\
&\ge \sum_{u=1}^3(\nu^{[n]}_{(f^u,H_i)}+\nu^{[n]}_{(f^u,H_{\sigma (i)})})-(2n+1)(\nu^{[1]}_{(f,H_i)}+\nu^{[1]}_{(f,H_{\sigma (i)})})+2\sum_{j=1}^q\nu^{[1]}_{(f,H_j)}.
\end{align*}
Summing-up both sides of the above inequality over all $1\le i\le q$, we have
$$\nu_{\prod_{i=1}^qP_i}\ge 2\sum_{u=1}^3\nu^{[n]}_{(f^u,H_i)}+(2q-4n-2)\sum_{j=1}^q\nu^{[1]}_{(f,H_j)}\ge\sum_{u=1}^3(2+\frac{2q-4n-2}{3n})\nu^{[n]}_{(f^u,H_i)}.$$
It is easy to see that $P_i\in B(1,0;f^1,f^2,f^3)\ (1\le i\le q)$, and hence 
$$\prod_{i=1}^qP_i\in B(q,0;f^1,f^2,f^3).$$ 
Then, by Proposition \ref{3.5} we have
$$ q\le n+1+3\rho\frac{n(n+1)}{2}+\frac{3nq}{2q+2n-2}.$$
This is a contradiction. 

Then $f^1\wedge f^2\wedge f^3\equiv 0$. The lemma is proved.
\end{proof}
\begin{lemma}\label{4.4} 
With the assumption of Theorem \ref{1.1}, let $h$ and $g$ be two elements of $\mathcal F(f,\{H_i\}_{i=1}^{q},1)$. If there exist a constant $\lambda$ and two indices $i,j$ such that $\frac{(h,H_i)}{(h,H_j)}=\lambda\frac{(g,H_i)}{(g,H_j)}$ then $\lambda =1$. 
\end{lemma}
{\textit{Proof.}\ Set $ H=  \dfrac{(h,H_{i})}{(h,H_j)}\text{ and }G=\dfrac{(g,H_{i})}{(g,H_j)}$.
Then $H=\lambda G$. Supposing that $\lambda\ne 1$, since $H=G$ on the set $\bigcup_{t\ne i,j}f^{-1}(H_t)\setminus (f^{-1}(H_i)\cup f^{-1}(H_j))$, we have 
$$\bigcup_{t\ne i,j}f^{-1}(H_t)\subset f^{-1}(H_i)\cup f^{-1}(H_j).$$
 By the assumption $\dim\left(f^{-1}(H_s)\cap f^{-1}(H_t)\right)\le m-2\ (s\ne t)$, this implies that
$$ \bigcup_{t\ne i,j}f^{-1}(H_t)=\emptyset.$$
Applying Proposition \ref{3.5} for the function $1\in B(0,0;f)$ and $(q-2)$ hyperplanes $\{H_t;t\ne i,j\}$, we have
$$ q-2\le n+1+\rho n(n+1).$$
This is a contradiction. Therefore $\lambda =1$. The lemma is proved. \hfill$\square$

\begin{lemma}\label{4.5} 
With the assumption of Theorem \ref{1.1}, let $f^1,f^2,f^3$ be three meromorphic mappings in $\mathcal F(f,\{H_i\}_{i=1}^{q},1)$. Suppose that $f^1\wedge f^2\wedge f^3\equiv 0$ and $V_i\sim V_j$ for some distinct indices $i$ and $j$. Then $f^1,f^2,f^2$ are not distinct. 
\end{lemma}
\textit{Proof.}\quad Suppose $f^1,f^2,f^2$ are distinct. Since $V_i\sim V_j$, we may suppose that $F_1^{ij}=F_2^{ij}\ne F_3^{ij}$. Since $f^1\wedge f^2\wedge f^3\equiv 0$ and $f^1\ne f^2$, there exists a meromorphic function $\alpha$ such that 
$$ F_3^{tj}=\alpha F_1^{tj}+(1-\alpha)F_2^{tj}\ (1\le t\le 2n+2). $$
This implies that $F_3^{ij}=F_1^{ij}=F^{ij}_2$. This is a contradiction. Hence $f^1,f^2,f^3$ are not distinct. The lemma is proved \hfill$\square$

\begin{lemma}\label{4.6} 
With the assumption of Theorem \ref{1.1}, let $f^1,f^2,f^3$ be three meromorphic mappings in $\mathcal F(f,\{H_i\}_{i=1}^{q},1)$. Suppose that $f^1,f^2,f^3$ are distinct and there are two indices $i, j \in \{1,\ldots ,q\} \ (i\ne j)$ such that $V_i\not\cong V_j$ and
$$\Phi_{ij}^{\alpha}:= \Phi^{\alpha}(F_1^{ij},F_2^{ij},F_3^{ij})\equiv 0$$
for every $\alpha =(\alpha_1,\ldots ,\alpha_m)\in\N^m$ with $|\alpha|=1$. Then for every $t\in\{1,\ldots ,q\}\setminus \{i\}$, the following assertions hold:
\begin{itemize}
\item[(i)] $\Phi^{\alpha}_{it}\equiv 0$ for all $|\alpha| \le 1.$
\item[(ii)] If $V_i\not\cong V_t$ then  $F_1^{ti},F_2^{ti},F_3^{ti}$ are distinct and there exists a meromorphic function $h_{it}\in S(1;f^1,f^2,f^3)$ such that
\begin{align*}
\nu_{h_{it}}\ge -\nu^{[1]}_{(f,H_i)}-\nu^{[1]}_{(f,H_t)}+\sum_{j\ne i,t}\nu^{[1]}_{(f,H_j)}.
\end{align*} 
\end{itemize}
\end{lemma}
 \textit{Proof.} By the supposition $V_i\not\cong V_j$, we may assume that $F_2^{ji}-F_1^{ji}\ne 0$. 

(a) For all $\alpha\in \N^m$ with $|\alpha|=1$, we have $\Phi_{ij}^{\alpha}=0$, and hence
\begin{align*}
\mathcal {D}^{\alpha}\biggl(\dfrac {F_3^{ji}-F_1^{ji}}{F_2^{ji}-F_1^{ji}}\biggl)=&\dfrac{1}{(F_2^{ji}-F_1^{ji})^2}\cdot\biggl ((F_2^{ji}-F_1^{ji})\cdot \mathcal {D}^{\alpha}(F_3^{ji}-F_1^{ji})\\
&\hspace{90pt}-(F_3^{ji}-F_1^{ji})\cdot \mathcal {D}^{\alpha}(F_2^{ji}-F_1^{ji})\biggl)\\
=&\dfrac{1}{{(F_2^{ji}-F_1^{ji})^2}}\cdot\left | 
\begin {array}{cccc}
1&1&1\\
F_1^{ji}&F_2^{ji} &F_3^{ji}\\
\mathcal {D}^{\alpha}(F_1^{ji}) &\mathcal {D}^{\alpha}(F_2^{ji}) &\mathcal {D}^{\alpha}(F_3^{ji})
\end {array}
\right| =0.
\end{align*}
Since the above equality holds for all $|\alpha|=1$, then there exists a constant $c\in \C$ such that
\begin{align*}
\dfrac {F_3^{ji}-F_1^{ji}}{F_2^{ji}-F_1^{ji}}=c.
\end{align*}
By Lemma \ref{4.3}, we have $f^1\wedge f^2\wedge f^3=0.$ Then for each index $t\in\{1,\ldots ,q\}\setminus\{i,j\}$, we have
\begin{align*}
0&= \det\left (
\begin{array}{ccc}
(f_1,H_i)&(f_1,H_j)&(f_1,H_{t})\\ 
(f_2,H_i)&(f_2,H_j)&(f_2,H_{t})\\
(f_3,H_i)&(f_3,H_j)&(f_3,H_{t})
\end{array}\right )
=\prod_{u=1}^3(f^u,H_i)\cdot\det\left (
\begin{array}{ccc}
1&F_1^{ji}&F_1^{ti}\\
1&F_2^{ji}&F_2^{ti}\\
1&F_3^{ji}&F_3^{ti}\\
\end{array}\right )\\
&=\prod_{u=1}^3(f^u,H_i)\cdot\det\left (
\begin{array}{ccc} 
F_2^{ji}-F_1^{ji}&F_2^{ti}-F_1^{ti}\\
F_3^{ji}-F_1^{ji}&F_3^{ti}-F_1^{ti}\\
\end{array}\right ).
\end{align*}
Thus
$$ (F_2^{ji}-F_1^{ji})\cdot (F_3^{ti}-F_1^{ti})= (F_3^{ji}-F_1^{ji})\cdot (F_2^{ti}-F_1^{ti}).$$
If $F_2^{ti}-F_1^{ti}=0$ then $F_3^{ti}-F_1^{ti}=0$, and hence $\Phi^{\alpha}_{it}=0$ for all $\alpha\in\N^m$ with $|\alpha|<1$. Otherwise, we have
$$ \dfrac{F_3^{ti}-F_1^{ti}}{F_2^{ti}-F_1^{ti}}=\dfrac{F_3^{ji}-F_1^{ji}}{F_2^{ji}-F_1^{ji}}=c.$$
This also implies that
\begin{align*}
\Phi^{\alpha}_{it}&=F_1^{it}\cdot F_2^{it} \cdot F_3^{it}\cdot
\left |
\begin{array}{ccc}
1&1&1\\
F_1^{ti}&F_2^{ti} &F_3^{ti}\\
\mathcal {D}^{\alpha}(F_1^{ti}) &\mathcal {D}^{\alpha}(F_2^{ti}) &\mathcal {D}^{\alpha}(F_3^{ti})\\
\end {array}
\right|\\
&=F_1^{it}\cdot F_2^{it} \cdot F_3^{it}\cdot
\left |
\begin{array}{cc}
F_2^{ti}-F_1^{ti} &F_3^{ti}-F_1^{ti}\\
\mathcal {D}^{\alpha}(F_2^{ti}-F_1^{ti})  &\mathcal {D}^{\alpha}(F_3^{ti}-F_1^{ti})\\
\end {array}
\right|\\
&=F_1^{it}\cdot F_2^{it} \cdot F_3^{it}\cdot
\left |
\begin{array}{cc}
F_2^{ti}-F_1^{ti} &c(F_2^{ti}-F_1^{ti})\\
\mathcal {D}^{\alpha}(F_2^{ti}-F_1^{ti})  &c\mathcal {D}^{\alpha}(F_2^{ti}-F_1^{ti})
\end {array}
\right|=0.
\end{align*}
Then one always has $\Phi^{\alpha}_{it}=0$ for all $t\in\{1,\ldots ,q\}\setminus \{i\}$. The first assertion is proved.

(b) We suppose that $V_i\not\cong V_t$. From the above part, we have
$$ c F_2^{si}+(1-c)F_1^{si}=F_3^{si}\ (s\ne i).$$
By the supposition that $f^1,f^2,f^3$ are distinct, we have $c\not\in\{0,1\}$. This implies that $F_1^{ti},F_2^{ti},F_3^{ti}$ are distinct. 

We consider the meromorphic mapping $G_{it}$ of $\B^m(1)$ into $ \P^1(\C)$ with a reduced representation 
$$G_{it}=((f^1,H_t)(f^2,H_i)(f_3,H_i)/h:(f^1,H_i)(f^2,H_t)(f_3,H_i)/h),$$
where $h$ is a holomorphic function on $\B^m(1)$. It is clear that
$$\|G_{it}\|\le\frac{\|f^1\|\cdot\|f^2\|\cdot\|f^3\|}{|h|}$$
and hence
$$T_{G_{it}}(r,r_0)\le T_{f_1}(r,r_0)+T_{f_2}(r,r_0)+T_{f_3}(r,r_0)-N_{h}(r,r_0).$$
This yields that $S(l;G_{it})\subset S(l;f^1,f^2,f^3)$ for all non-negative integers $l$.

For a point $z\not\in I(G_{it})$ which is a zero of one of 
$$\{(f^1,H_t)(f^2,H_i)(f_3,H_i)/h,(f^1,H_i)(f^2,H_t)(f_3,H_i)/h,(f^1,H_i)(f^2,H_i)(f_3,H_t)/h\},$$ 
then $z$ must be either zero of $(f,H_i)$ or zero of $(f,H_t)$, and hence
\begin{align*}
\nu^{[1]}_{(f^1,H_t)(f^2,H_i)(f_3,H_i)/h}(z)&+\nu^{[1]}_{(f^1,H_i)(f^2,H_t)(f_3,H_i)/h}(z)+\nu^{[1]}_{(f^1,H_i)(f^2,H_i)(f_3,H_t)/h}(z)\\
&=1\le \nu^{[1]}_{(f,H_i)}(z)+\nu^{[1]}_{(f,H_t)}(z).
\end{align*}
On the other hand, its is clear that
\begin{align}\label{4.7}
\nu^{[1]}_{(f^1,H_t)(f^2,H_i)(f_3,H_i)/h-(f^1,H_i)(f^2,H_t)(f_3,H_i)/h}\ge \sum_{\underset{v\ne i,t}{v=1}}^{q}\nu^{[1]}_{(f,H_v)}.
\end{align}
By Lemma \ref{4.3}, we see that $G_{it}$ is not constant. Then there is $\beta_{it}=(\beta_{it0},\beta_{it1})$, where $\beta_{itk}\in\N^m$ and $|\beta_{itk}|\le k\ (k=0,1)$, such that
$$ D^{\beta_{it}}=\left |
\begin{array}{cc}
\mathcal D^{\beta_{it0}}((f^1,H_t)(f^2,H_i)(f_3,H_i))&\mathcal D^{\beta_{it0}}((f^1,H_i)(f^2,H_t)(f_3,H_i))\\ 
\mathcal D^{\beta_{it1}}((f^1,H_t)(f^2,H_i)(f_3,H_i))&\mathcal D^{\beta_{it1}}((f^1,H_i)(f^2,H_t)(f_3,H_i))
\end{array}
\right|\not\equiv 0 $$
We put
$$ h_{it}=\frac{\left (
	(f^1,H_t)(f^2,H_i)(f_3,H_i)-(f^1,H_i)(f^2,H_t)(f_3,H_i)\right )D^{\beta_{it}}}{(f^1,H_t)(f^2,H_i)(f^3,H_i).(f^1,H_i)(f^2,H_t)(f_3,H_i).(f^1,H_i)(f^2,H_i)(f^3,H_t)}. $$
Hence, we see that $h_{it}\in S(1;G_{it})\subset S(1;f^1,f^2,f^3).$
Also by usual argument, we have
\begin{align*}
\nu_{h_{it}}&=\nu_{(F_1^{ti}-F_2^{ti})h_t}+\nu_W-\sum_{u=1}^3\nu_{F_u^{ti}h_t}\\
&\ge\sum_{\underset{v\ne i,t}{v=1}}^{q}\nu^{[1]}_{(f,H_v)}-\sum_{u=1}^3\nu^{[1]}_{F_u^{ti}h_t}\\
&\ge\sum_{\underset{v\ne i,t}{v=1}}^{q}\nu^{[1]}_{(f,H_v)}- \nu^{[1]}_{(f,H_i)}-\nu^{[1]}_{(f,H_t)}.
\end{align*}
Moreover, we have
$|h_{it}|\in S(0;G_{it})\subset S(0;f^1,f^2,f^3).$
The lemma is proved.\hfill$\square$

\begin{lemma}\label{4.8}
With the assumption of Theorem \ref{1.1}, let $f^1,f^2,f^3$ be three meromorphic mappings in $\mathcal F(f,\{H_i\}_{i=1}^{q},1)$.
Assume that there exist $i, j \in \{1,2,\ldots ,q\} \ (i\ne j)$ and  $\alpha\in\N^m$ with $|\alpha|=1$ such that
$\Phi^{\alpha}_{ij}\not\equiv 0.$
Then there exists a holomophic function $g_{ij}\in B(1,1;f^1,f^2,f^3)$ such that
\begin{align*}
\nu_{g_{ij}}\ge \sum_{u=1}^{3}\nu^{[n]}_{(f^u,H_i)}+\sum_{u=1}^3\nu^{[n]}_{(f^u,H_j)}+2\sum_{\overset{t=1}{t\ne i,j}}^{q}\nu^{[1]}_{(f,H_t)}-(2n+1)\nu^{[1]}_{(f,H_i)}-(n+1)\nu^{[1]}_{(f,H_j)}+\nu_j
\end{align*}
\end{lemma}
\begin{proof}
We have
\begin{align}\label{4.9}
\begin{split}
\Phi^{\alpha}_{ij}
&=F_1^{ij}\cdot F_2^{ij}\cdot F_3^{ij}\cdot 
\left | 
\begin {array}{cccc}
1&1&1\\
F_1^{ji}&F_2^{ji} &F_3^{ji}\\
\mathcal {D}^{\alpha}(F_1^{ji}) &\mathcal {D}^{\alpha}(F_2^{ji}) &\mathcal {D}^{\alpha}(F_3^{ji})\\
\end {array}
\right|\\
&=\left | 
\begin {array}{cccc}
F_1^{ij}&F_2^{ij} &F_3^{ij}\\
1&1&1\\
F_1^{ij}\mathcal {D}^{\alpha}(F_2^{ji}) &F_2^{ij}\mathcal {D}^{\alpha}(F_2^{ji}) &F_3^{ij}\mathcal {D}^{\alpha}(F_3^{ji})
\end {array}
\right|\\
&=F_1^{ij}\bigl(\dfrac{\mathcal {D}^{\alpha}(F_3^{ji})}{F^{ji}_3}-\dfrac{\mathcal {D}^{\alpha}(F_2^{ji})}{F^{ji}_2}\bigl)
+F^{ij}_2\bigl(\frac{\mathcal {D}^{\alpha}(F_1^{ji})}{F^{ji}_1}-\frac{\mathcal {D}^{\alpha}(F_3^{ji})}{F^{ji}_3}\bigl)\\
&\ \ +F^{ij}_3\bigl(\frac{\mathcal {D}^{\alpha}(F_2^{ji})}{F^{ji}_2}-\frac{\mathcal {D}^{\alpha}(F_1^{ji})}{F^{ji}_1}\bigl).
\end{split}
\end{align}
This implies that
$$(\prod_{u=1}^3(f^u,H_j))\cdot\Phi^{\alpha}_{ij}=g_{ij},$$
where
\begin{align*}
g_{ij}=&(f^1,H_i)\cdot (f^2,H_j)\cdot (f^3,H_j)\cdot\left (\dfrac{\mathcal {D}^{\alpha}(F_3^{ji})}{F^{ji}_3}-\dfrac{\mathcal {D}^{\alpha}(F_2^{ji})}{F^{ji}_2}\right)\\
&+(f^1,H_j)\cdot (f^2,H_i)\cdot (f^3,H_j)\cdot\left (\dfrac{\mathcal {D}^{\alpha}(F_1^{ji})}{F^{ji}_1}-\dfrac{\mathcal {D}^{\alpha}(F_3^{ji})}{F^{ji}_3}\right)\\
&+(f^1,H_j)\cdot (f^2,H_j)\cdot (f^3,H_i)\cdot\left (\dfrac{\mathcal {D}^{\alpha}(F_2^{ji})}{F^{ji}_2}-\dfrac{\mathcal {D}^{\alpha}(F_1^{ji})}{F^{ji}_1}\right).
\end{align*}
Hence, we easily see that
$$|g_{ij}|\le C\cdot \|f^1\|\cdot \|f^2\|\cdot \|f^3\|\cdot\sum_{u=1}^3\left|\dfrac{\mathcal {D}^{\alpha}(F_u^{ji})}{F^{ji}_u}\right|,$$
where $C$ is a positive constant, and then $g_{ij}\in B(1;1;f^1,f^2,f^3)$.
It is clear that
\begin{align}\label{4.10}
\nu_{\Phi^{\alpha}_{ij}}=-\sum_{u=1}^3\nu_{(f^u,H_j)}+\nu_{g_{ij}}.
\end{align}

Hence, it is sufficient for us to prove that
\begin{align}\label{4.11} 
\begin{split}
\nu_{\Phi^\alpha_{ij}}\ge &-\sum_{u=1}^3\nu_{(f^u,H_j)}+\sum_{u=1}^{3}\nu^{[n]}_{(f^u,H_i)}+\sum_{u=1}^3\nu^{[n]}_{(f^u,H_j)}\\
&+2\sum_{\overset{t=1}{t\ne i,j}}^{q}\nu^{[1]}_{(f,H_t)}-(2n+1)\nu^{[1]}_{(f,H_i)}-(n+1)\nu^{[1]}_{(f,H_j)}+\nu_j.
\end{split}
\end{align}
We set 
$$S=\bigcup_{s\ne t}\{z;\nu_{(f,H_s)}(z)\cdot\nu_{(f,H_t)}(z)>0\}.$$
Then $S$ is an analytic subset of codimension at least two in $\B^m(1)$. 
We denote by $P$ the right hand side of the inequlity (\ref{4.11}).
In order to prove the inequality (\ref{4.11}), it is sufficient for us to show that 
\begin{align}\label{4.12}
\nu_{\Phi^\alpha_{ij}}(z)\ge P(z)
\end{align}
for all $z$ outside the set $I$. 

Indeed, for $z\not\in I$, we distinguish the following cases:

\vskip0.1cm
\noindent
\textit{Case 1:}  $z\in\supp\nu_{(f,H_t)}\ (t\ne i,j)$. We see that $P(z)=2$.  We write $\Phi^\alpha_{ij}$ in the form
$$ \Phi^{\alpha}_{ij}
=F_1^{ij}\cdot F_2^{ij}\cdot F_3^{ij}
\times\left | 
\begin {array}{cccc}
 \bigl (F_1^{ji}-F_{2}^{ji}\bigl ) & \bigl (F_1^{ji}-F_{3}^{ji}\bigl )\\
\mathcal{D}^{\alpha}\bigl ( F_1^{ji}-F_{2}^{ji}\bigl ) & \mathcal{D}^{\alpha}\bigl ( F_1^{ji}-F_{3}^{ji}\bigl )
\end {array}
\right|. $$
Then  by the assumption that $f^1,f^2,f^3$ coincide on $T_t$, we have $F_1^{ji}=F_2^{ji}=F_3^{ji}$ on $T_t$. The property of the general Wronskian implies that $\nu_{\Phi^{\alpha}_{ij}}(z)\ge 2=P(z)$.

\textit{Case 2:} $z\in \supp\nu_{(f,H_i)}$. We have 
$$P(z)= \sum_{u=1}^3\nu^{[n]}_{(f^u,H_i)}(z)-(2n+1)\le\min_{1\le u\le 3}\{\nu^{[n]}_{(f^u,H_i)}(z)\}-1.$$
We may assume that $\nu_{(f^1,H_{i})}(z)\le\nu_{(f^2,H_{i})}(z)\le\nu_{(f^3,H_{i})}(z)$. We write
\begin{align*}
\Phi^{\alpha}_{ij}
=F_1^{ij}\biggl [F_2^{ij}(F_1^{ji}-F_{2}^{ji})F_3^{ij}\mathcal{D}^{\alpha} (F_1^{ji}-F_{3}^{ji})-F_3^{ij}(F_1^{ji}-F_{3}^{ji})F_2^{ij}\mathcal{D}^{\alpha} (F_1^{ji}-F_{2}^{ji})\biggl ]
\end{align*}
It is easy to see that $F_2^{ij}(F_1^{ji}-F_{2}^{ji})$, $F_3^{ij}(F_1^{ji}-F_{3}^{ji})$ are holomorphic on a neighborhood of $z$, and 
\begin{align*}
&\nu^{\infty}_{F_3^{ij}\mathcal{D}^{\alpha}(F_1^{ji}-F_{3}^{ji})}(z)\le 1, \\
\text{ and }\ \ \ &\nu^{\infty}_{F_2^{ij}\mathcal{D}^{\alpha} (F_1^{ji}-F_{2}^{ji})}(z)\le 1.
\end{align*}
Therefore, it implies that
\begin{align*}
\nu_{\Phi^{\alpha}_{ij}}(z)&\ge \nu^{[n]}_{(f^1,H_i)}(z)-1\ge P(z).
\end{align*}

\noindent
\textit{Case 3:} $z\in \supp\nu_{(f,H_j)}$.  We may assume that 
$$\nu_{F_1^{ji}}(z)=d_1\ge \nu_{F_2^{ji}}(z)=d_2\ge\nu_{F_3^{ji}}(z)=d_3.$$
 Choose a holomorphic function $h$ on $\B^m(1)$ whose the multiplicity of zero at $z$ equal to $1$ such that
$F_u^{ji}=h^{d_u}\varphi_u\ (1\le u\le 3),$ where $\varphi_u$ are meromorphic on $\B^m(1)$ and 
holomorphic on a neighborhood of $z$. Then
\begin{align*}
 \Phi^{\alpha}_{ij}&=F_1^{ij}\cdot F_2^{ij}\cdot F_3^{ij}\cdot 
\left | 
\begin {array}{ccc}
F_2^{ji}-F_1^{ji}&F_3^{ji}-F_1^{ji}\\
\mathcal {D}^{\alpha}(F_2^{ji}-F_1^{ji}) &\mathcal {D}^{\alpha}(F_3^{ji}-F_1^{ji})\\
\end {array}
\right|\\
&= F_1^{ij}\cdot F_2^{ij}\cdot F_3^{ij}\cdot h^{d_2+d_3}\cdot 
\left | 
\begin {array}{ccc}
\varphi_2-h^{d_1-d_2}\varphi_1&\varphi_3-h^{d_1-d_3}\varphi_1\\
\dfrac{\mathcal {D}^{\alpha}(h^{d_2-d_3}\varphi_2-h^{d_1-d_3}\varphi_1)}{h^{d_2-d_3}} &\mathcal {D}^{\alpha}(\varphi_3-h^{d_1-d_3}\varphi_1)\\
\end {array}
\right|.
\end{align*}
This yields that
\begin{align*}
\nu_{\Phi^\alpha_{ij}}(z)&\ge -\sum_{u=1}^3 \nu_{(f^u,H_j)}(z)+d_2+d_3-\max\{0,\min\{1,d_2-d_3\}\}.
\end{align*}
On the other hand
\begin{align*}
P(z)&=-\sum_{u=1}^3 \nu_{(f^u,H_j)}(z)+\sum_{u=1}^3\min\{n,d_u\}-(n+1)+\chi_{\nu_j}\\
&=-\sum_{u=1}^3 \nu_{(f^u,H_j)}(z)+d_2+d_3-1+\chi_{\nu_j}\le \nu_{\Phi^\alpha_{ij}}(z).
\end{align*}

From the above three cases, the inequality (\ref{4.12}) holds. The lemma is proved.
\end{proof}

\begin{proof}[{\sc Proof of theorem \ref{1.1}}]  Suppose contrarily that there exist three distinct meromorphic mappings $f^1,f^2,f^3$ in $\mathcal F(f,\{H_i\}_{i=1}^{q},1)$.  By Lemma \ref{4.3}, we have $f^1\wedge f^2\wedge f^3\equiv 0$. Without loss of generality, we may assume that
$$\underbrace{V_1\cong \cdots\cong V_{l_1}}_{\text { group } 1}\not\cong
\underbrace{V_{l_1+1}\cong\cdots\cong V_{l_2}}_{\text { group } 2}\not\equiv \underbrace{V_{l_2+1}\cong \cdots\cong V_{l_3}}_{\text { group } 3}\not\cong\cdots \not\cong\underbrace{V_{l_{s-1}}\cong\cdots\cong V_{l_{s}}}_{\text { group } s},$$
where $l_s=q.$ 

Denote by $P$ the set of all $i\in\{1,\ldots ,q\}$ satisfying there exist $j\in\{1,\ldots ,q\}\setminus\{i\}$ such that $V_i\not\cong V_j$ and $\Phi^{\alpha}_{ij}\equiv 0$ for all $\alpha\in\Z^m_+$ with $|\alpha|\le 1.$ We consider the following three cases.

\noindent
\textit{Case 1:} $\sharp P\ge 2$. Then $P$ contains two elements $i,j$. Then we have $\Phi^{\alpha}_{ij}=\Phi^{\alpha}_{ji}=0$ for all $\alpha\in\Z^m_+$ with $|\alpha|\le 1.$ By Lemma \ref{2.3}, there exist two functions, for instance they are $F_1^{ij}$ and $F^2_{ij}$, and a constant $\lambda$ such that $F_1^{ij}=\lambda F^2_{ij}$. This yields that $F_1^{ij}=F^{ij}_2$ (by Lemma \ref{4.4}). Then by Lemma \ref{4.6} (ii), we easily see that $V_i\cong V_j$, i.e., $V_i$ and $V_j$ belong to the same group in the above partition. 

Without loss of generality, we may assume that $i=1$ and $j=2$. Since $f^1,f^2,f^3$ are supposed to be distinct, the number of each group in the above partition is less than $n+1$. Hence we have $V_1\cong V_2\not\cong V_t$ for all $t\in\{n+1,\ldots ,q\}$. By Lemma \ref{4.6} (ii), we have
\begin{align*}
\nu_{h_{1t}}&\ge\sum_{s\ne 1,t}\nu^{[1]}_{(f,H_s)}- \nu^{[1]}_{(f,H_1)}(r)-\nu^{[1]}_{(f,H_t)}\\
 \text{and }\nu_{h_{2t}}&\ge\sum_{s\ne 2,t}\nu^{[1]}_{(f,H_s)}- \nu^{[1]}_{(f,H_2)}(r)-\nu^{[1]}_{(f,H_t)}.
\end{align*}
Summing-up both sides of the above two inequalities, we get
\begin{align*}
\nu_{h_{1t}}+\nu_{h_{2t}}&\ge\sum_{s\ne 1,2,t}\nu^{[1]}_{(f,H_s)}- 2\nu^{[1]}_{(f,H_t)}.
\end{align*}
After summing-up both sides of the above inequalities over all $t\in\{n+1,\ldots,q\}$, we easily obtain
\begin{align*}
\sum_{i=n+1}^{q}(\nu_{h_{1t}}+\nu_{h_{2t}})&\ge (n+2)\sum_{t=3}^n\nu^{[1]}_{(f,H_t)}+n\sum_{t=n+1}^{q}\nu^{[1]}_{(f,H_t)}\\
&\ge n\sum_{t=3}^{q}\nu^{[1]}_{(f,H_t)}\ge\frac{1}{3}\sum_{u=1}^3\sum_{t=3}^{q}\nu^{[n]}_{(f^u,H_t)}.
\end{align*}
This implies that
$$ \nu_{\prod_{i=n+1}^{q}(h_{1t}h_{2t})}\ge\frac{1}{3}\sum_{u=1}^3\sum_{t=3}^{q}\nu^{[n]}_{(f^u,H_t)}. $$
Since $\prod_{i=n+1}^{q}(h_{1t}h_{2t})\in B(0,2(q-n);f^1,f^2,f^3)$, from Proposition \ref{3.5} we have
$$ q\le n+1+ 3\rho\frac{n(n+1)}{2}+3\rho 2(q-n).$$
This is a contradiction.

\vskip0.2cm 
\noindent
{\textit{Case 2:}} $\sharp P=1$. We assume that $P=\{1\}$. We easily see that $V_{1}\not\cong V_i$ for all $i=2,\ldots ,q$ (otherwise $i\in P$, this contradicts to $\sharp P=1$). Then by Lemma \ref{4.6} (ii), we have
\begin{align*}
\nu_{h_{1i}}\ge\sum_{s\ne 1,t}\nu^{[1]}_{(f,H_s)}- \nu^{[1]}_{(f,H_1)}(r)-\nu^{[1]}_{(f,H_i)}.
\end{align*}
Summing-up both sides of the above inequality over all $i=2,\ldots ,q$, we get
\begin{align}\label{4.13}
\sum_{i=2}^{q}\nu_{h_{1i}}\ge& (q-3)\sum_{i=2}^{q}\nu^{[1]}_{(f,H_i)}-(q-1)\nu^{[1]}_{(f,H_1)}.
\end{align}
We also see that $i\not\in P$ for all $2\le i\le q$. Set
$$ \sigma (i)=\begin{cases}
i+n&\text{ if }i+n\le q,\\
i+n-q+1&\text{ if } i+n>q.
\end{cases}$$
Then $i$ and $\sigma (i)$ belong to two distinct groups, i.e., $V_i\not\cong V_{\sigma (i)}$, for all $i\in\{2,\ldots ,q\}$, and hence $\Phi^{\alpha}_{i\sigma (i)}\not\equiv 0$ for some $\alpha\in\N^m$ with $|\alpha|\le 1$. By Lemma \ref{4.8} we have
\begin{align*}
\nu_{g_{i\sigma(i)}}\ge \sum_{u=1}^{3}\sum_{t=i,\sigma (i)}\nu^{[n]}_{(f^u,H_t)}-(2n+1)\nu^{[1]}_{(f,H_i)}-(n+1)\nu^{[1]}_{(f,H_{\sigma (i)})}+2\sum_{\overset{t=1}{t\ne i,{\sigma (i)}}}^{q}\nu_{(f,H_t)}^{[1]}.
\end{align*}
Summing-up both sides of the above inequality over all $i=2,\ldots ,q$, we get
\begin{align*}
\sum_{i=2}^{q}\nu_{g_{i\sigma(i)}}&\ge 2\sum_{i=2}^{q}\sum_{u=1}^{3}\nu^{[n]}_{(f^u,H_i)}+(2q-3n-8)\sum_{i=2}^{q}\nu^{[1]}_{(f,H_i)}+2(q-1)\nu^{[1]}_{(f,H_1)}\\
&\ge 2\sum_{i=2}^{q}\sum_{u=1}^{3}\nu^{[n]}_{(f^u,H_i)}+\dfrac{4q-3n-14}{3}\sum_{u=1}^3\sum_{i=2}^{q}\nu^{[1]}_{(f^u,H_i)}-2\sum_{i=2}^{q}\nu_{h_{1i}}\\
&\ge \dfrac{4q+3n-14}{3n}\sum_{u=1}^{3}\sum_{i=2}^{q}\nu^{[n]}_{(f^u,H_i)}-2\sum_{i=2}^{q}\nu_{h_{1i}}.
\end{align*}
This implies that
$$ \nu_{\prod_{i=2}^{q}(g_{i\sigma(i)}h_{1i}^2)}\ge \dfrac{4q+3n-14}{3n}\sum_{u=1}^{3}\sum_{i=2}^{q}\nu^{[n]}_{(f^u,H_i)}\ge \dfrac{11n-6}{3n}\sum_{u=1}^{3}\sum_{i=2}^{q}\nu^{[n]}_{(f^u,H_i)}.$$
It is clear that $\prod_{i=2}^{q}(g_{i\sigma(i)}h_{1i}^2)\in B(q-1,q-1;f^1,f^2,f^3)$. Then, from Proposition \ref{3.5}, we have
$$ q\le n+1+3\rho\frac{n(n+2)}{2}+\frac{3n}{11n-6}((q-1)+\rho(q-1)).$$
This is a contradiction. 

\noindent
{\textit{Case 3:}} $P=\emptyset$. Then for all $i\ne j$, by Lemma \ref{4.8} we have
\begin{align*}
\nu_{g_{ij}}&\ge \sum_{u=1}^{3}\nu^{[n]}_{(f^u,H_i)}+\sum_{k=1}^3 \nu^{[n]}_{(f^k,H_j)}+2\sum_{\overset{t=1}{t\ne i,j}}^{q}\nu_{(f,H_t)}^{[1]}\\
&\ \ -(2n+1)\nu^{[1]}_{(f,H_i)}-(n+1)\nu^{[1]}_{(f,H_j)}+\nu_j.
\end{align*}
Setting 
$$\gamma (i)=\begin{cases}
i+n&\text{ if }i\le q-n\\
i+n-q&\text{ if }i>q-n
\end{cases}$$
and summing-up both sides of the above inequality over all pairs $(i,\gamma (i))$, we get
\begin{align}\label{4.14}
\sum_{i=1}^q\nu_{g_{i\gamma (i)}}\ge &2\sum_{u=1}^3\sum_{t=1}^{q}\nu^{[n]}_{(f^u,H_t)}+(2q-3n-6)\sum_{t=1}^{q}\nu^{[1]}_{(f,H_t)}+\sum_{t=1}^{q}\nu_t.
\end{align}

On the other hand, by Lemma \ref{4.5}, we see that $V_j\not\sim V_l$ for all $j\ne l$. Hence, we have
$$P_{st}^{i\gamma (i)}:\overset{Def}=(f^s,H_i)(f^t,H_{\gamma(i)})-(f^t,H_{\gamma(i)})(f^s,H_i)\not\equiv 0\ (s\ne t, 1\le i\le q).$$
\begin{claim}\label{4.15}
With $i\ne j\ne\sigma (i)$, for every $z\in f^{-1}(H_j)$ we have 
$$\sum_{1\le s<t\le 3}\nu_{P_{st}^{i\gamma (i)}}(z)\ge 4\nu_{(f,H_j)}^{[1]}(z)-\nu_j(z).$$ 
\end{claim}
Indeed, for $z\in f^{-1}(H_j)\cap\sup_{\nu_j}$, we have
$$ 4\nu_{(f,H_j)}^{[1]}(z)-\nu_j(z)\le 4-1=3\le  \sum_{1\le s<t\le 3}\nu_{P_{st}^{i\gamma (i)}}(z).$$
Now, for $z\in f^{-1}(H_j)\setminus\sup{\nu_j}$, we may assume that $\nu_{(f^1,H_j)}(z)<\nu_{(f^2,H_j)}(z)\le\nu_{(f^3,H_j)}(z)$. Since $f^1\wedge f^2\wedge f^3\equiv 0$, we have $\det (V_i,V_{\gamma (i)},V_j)\equiv 0$, and hence
\begin{align*}
(f^1,H_j)P_{23}^{i\gamma (i)}=(f^2,H_j)P_{13}^{i\gamma (i)}-(f^3,H_j)P_{12}^{i\gamma (i)}.
\end{align*}
This yields that 
$$ \nu_{P_{23}^{i\gamma (i)}}(z)\ge 2 $$
and hence 
$$\sum_{1\le s<t\le 3}\nu_{P_{st}^{i\gamma (i)}}(z)\ge 4=4\nu_{(f,H_j)}^{[1]}(z)-\nu_j(z).$$
Hence, the claim is proved.

\vskip0.2cm 
On the other hand, with $j=i$ or $j=\sigma(i)$, for every $z\in f^{-1}(H_j)$ we see that 
\begin{align*}
\nu_{P_{st}^{i\gamma (i)}}(z)\ge&\min\{\nu_{(f^s,H_j)}(z),\nu_{(f^t,H_j)}(z)\}\\
\ge& \nu^{[n]}_{(f^s,H_j)}(z)+\nu^{[n]}_{(f^t,H_j)}(z)-n \nu^{[1]}_{(f,H_j)}(z).
\end{align*}
$$\text{and hence }\sum_{1\le s<t\le 3}\nu_{P_{st}^{i\gamma (i)}}(z)\ge 2\sum_{u=1}^3\nu^{[n]}_{(f^u,H_j)}(z)-3n\nu^{[1]}_{(f,H_j)}(z).$$
Combining this inequality and the above claim, we have
$$\sum_{1\le s<t\le 3}\nu_{P_{st}^{i\gamma (i)}}(z)\ge \sum_{j=i,\gamma (i)}\left (2\sum_{u=1}^3\nu^{[n]}_{(f^u,H_j)}(z)-3n\nu^{[1]}_{(f,H_j)}(z)\right )+\sum_{\overset{j=1}{j\ne i,\gamma (i)}}^q(4\nu^{[1]}_{(f,H_j)}(z)-\nu_j(z)).$$
This yields that
\begin{align}\label{4.16}
\sum_{1\le s<t\le 3}\nu_{P_{st}^{i\gamma (i)}}\ge &\sum_{j=i,\gamma (i)}\left (2\sum_{u=1}^3\nu^{[n]}_{(f^u,H_j)}-3n\nu^{[1]}_{(f,H_j)}\right)+\sum_{\overset{j=1}{j\ne i,\gamma (i)}}^q(4\nu^{[1]}_{(f,H_j)}-\nu_j).
\end{align}
On the other hand, we easily see that $\prod_{1\le s<t\le 3}P_{st}^{i\gamma (i)}\in B(2,0;f^1,f^2,f^3)$.

Summing-up both sides of the above inequality over all $i$, we obtain
\begin{align*}
\sum_{i=1}^q\sum_{1\le s<t\le 3}\nu_{P_{st}^{i\gamma (i)}}\ge &4\sum_{u=1}^3\sum_{i=1}^{q}\nu^{[n]}_{(f^u,H_i)}+(4q-6n-8)\sum_{i=1}^{q}\nu^{[1]}_{(f,H_i)}-(q-2)\sum_{i=1}^{q}\nu_i.
\end{align*}
Thus
\begin{align*}
\sum_{i=1}^{q}\nu_i\ge&\dfrac{4}{q-2}\sum_{u=1}^3\sum_{i=1}^{q}\nu^{[n]}_{(f^u,H_i)}+\dfrac{4q-6n-8}{q-2}\sum_{i=1}^{q}\nu^{[1]}_{(f,H_i)}-\frac{1}{q-2}\sum_{i=1}^q\sum_{1\le s<t\le 3}\nu_{P_{st}^{i\gamma (i)}}.
\end{align*}
Using this estimate, from (\ref{4.14}) we have
\begin{align*}
\sum_{i=1}^q&\nu_{g_{i\gamma (i)}}+\frac{1}{q-2}\sum_{i=1}^q\sum_{1\le s<t\le 3}\nu_{P_{st}^{i\gamma (i)}}\\
&\ge \left (2+\dfrac{4}{q-2}\right)\sum_{u=1}^3\sum_{t=1}^{q}\nu^{[n]}_{(f^u,H_t)}+\left (2q-3n-6+\dfrac{4q-6n-8}{q-2}\right)\sum_{i=1}^{q}\nu^{[1]}_{(f^u,H_i)}\\
&\ge \left (2+\dfrac{4}{q-2}+\frac{n-2}{3n}+\frac{4q-6n-8}{3n(q-2)}\right)\sum_{u=1}^3\sum_{t=1}^{q}\nu^{[n]}_{(f^u,H_t)}.
\end{align*}
This yields that
$$\nu_{\prod_{i=1}^q(g_{i\gamma (i)}^{q-2}P_{12}^{i\gamma (i)}P_{13}^{i\gamma (i)}P_{23}^{i\gamma (i)})}\ge \left (2q+\frac{(n-2)(q-2)+4q-6n-8}{3n}\right)\sum_{u=1}^3\sum_{t=1}^{q}\nu^{[n]}_{(f^u,H_t)}.$$
We see that $\prod_{i=1}^q(g_{i\gamma (i)}^{q-2}P_{12}^{i\gamma (i)}P_{13}^{i\gamma (i)}P_{23}^{i\gamma (i)})\in B(q^2,q(q-2);f^1,f^2,f^3)$. Then, from Proposition \ref{3.5}, we have
$$ q\le n+1+3\rho\frac{n(n+1)}{2}+\frac{3n(q^2+\rho q(q-2))}{6nq+(n+2)(q-2)-6n}.$$
 This is a contradiction. 

Hence the supposition is false. Therefore, $\sharp\mathcal F(f,\{H_i\}_{i=1}^{q},1)\le 2$. We complete the proof of the theorem.
\end{proof}

\section{Proof of Theorem \ref{1.2}}
Since in the case where $M=\C^m$, the theorem has already proved by the author in \cite{SDQ11}, without loss of generality, in this proof we only consider the case where $M=\mathbb B^m(1)$.

 In order to prove Theorem \ref{1.2}, we firstly prove the following theorem, which is the generalization of the uniqueness theorem for meromorphic mappings of $\C^m$ into $\P^n(\C)$ sharing $2n+3$ hyperplanes in general position regardless of multiplicity.

\begin{theorem}\label{5.1}
Let $M$ be a complete connected K\"{a}hler manifold whose universal covering is biholomorphic to $\C^m$ or the unit ball $\B^m(1)$ of $\C^m$, and let $f$ be a linearly non-degenerate meromorphic maps of $M$ into $\P^n(\C)\ (n\ge 2)$. Let $H_1,\ldots,H_q$ be $q$ hyperplanes of $\P^n(\C)$ in general possition. Assume that $f$ satisfies the condition $(C_\rho)$ and
$$\dim f^{-1}(H_i)\cap f^{-1}(H_j) \le m-2 \quad (1 \le i<j \le q).$$ 
Then $\sharp \mathcal {F}(f,\{H_i\}_{i=1}^q,1)=1$ if $q>n+1+\rho n(n+1)+ \dfrac{2nq}{q+2n-2}$, in particular if $q>2n+2+2\rho n(n+1)$.
\end{theorem} 
\begin{proof}
Suppose contrarily that $\sharp \mathcal {F}(f,\{H_i\}_{i=1}^q,1)>1$. Then there exists two distinct elements $f^1,f^2$ of $\mathcal F(f,\{H_i\}_{i=1}^{q},1).$ 
By changing indices if necessary, we may assume that
$$\underbrace{\dfrac{(f^1,H_1)}{(f^2,H_1)}\equiv \dfrac{(f^1,H_2)}{(f^2,H_2)}\equiv \cdot\cdot\cdot\equiv \dfrac{(f^1,H_{k_1})}
	{(f^2,H_{k_1})}}_{\text { group } 1}\not\equiv
\underbrace{\dfrac{(f^1,H_{k_1+1})}{(f^2,H_{k_1+1})}\equiv \cdot\cdot\cdot\equiv \dfrac{(f^1,H_{k_2})}{(f^2,H_{k_2})}}_{\text { group } 2}$$
$$\not\equiv \underbrace{\dfrac{(f^1,H_{k_2+1})}{(f^2,H_{k_2+1})}\equiv \cdot\cdot\cdot\equiv \dfrac{(f^1,H_{k_3})}{(f^2,H_{k_3})}}_{\text { group } 3}\not\equiv \cdot\cdot\cdot\not\equiv \underbrace{\dfrac{(f^1,H_{k_{s-1}+1})}{(f^2,H_{k_{s-1}+1})}\equiv\cdot\cdot\cdot \equiv 
	\dfrac{(f^1,H_{k_s})}{(f^2,H_{k_s})}}_{\text { group } s},$$
where $k_s=q.$ 

Then, we have
$P_i=(f^1,H_i)(f^2,H_{\sigma (i)})-(f^2,H_i)(f^1,H_{\sigma (i)})\not\equiv 0$, 
for all $1\le i\le q$, where
\begin{equation*}
\sigma (i)=
\begin{cases}
i+n& \text{ if $i+n\leq q$},\\
i+n-q&\text{ if  $i+n> q$}.
\end{cases}
\end{equation*}

By using the same argument as in the proof of Theorem \ref{1.1}, we have
\begin{align*}
\nu_{P_i}(z) 
\ge&\sum_{v=i,\sigma (i)}\min\{\nu_{(f^1,H_v)}(z),\nu_{(f^2,H_v)}(z)\}+\sum_{\underset{v\ne i,\sigma (i)}{v=1}}^{q}\nu_{(f^1,H_v)}^{[1]}(z)\\
\ge&\sum_{v=i,\sigma (i)}\bigl (\nu_{(f^1,H_v)}^{[n]}(z)+\nu_{(f^2,H_v)}^{[n]}(z)-n\nu_{(f^1,H_v)}^{[1]}(z)\bigl )+\sum_{\underset{v\ne i,\sigma (i)}{v=1}}^{q}\nu_{(f^1,H_v)}^{[1]}(z).
\end{align*}
Summing-up of both sides of the above inequality over 
$i=1,\ldots,q$, we obtain 
\begin{align}\label{5.2}
\begin{split}
\nu_{\prod_{i=1}^qP_i}(z) \ge &\sum_{i=1}^{q}\biggl (\nu_{(f^1,H_i)}^{[n]}(z)+\nu_{(f^2,H_i)}^{[n]}(z)\biggl )+(q-2)\cdot \sum_{i=1}^{q}\nu_{(f^1,H_i)}^{[1]}(z)\\
= &\left (1+\frac{q-2}{2n}\right)\cdot \sum_{i=1}^{q}\left (\nu_{(f^1,H_i)}^{[n]}(z)+\nu_{(f^2,H_i)}^{[n]}(z)\right )
\end{split}
\end{align}
Since $\prod_{i=1}^qP_i\in B(q,0;f^1,f^2)$, from Proposition \ref{3.5} we have
$$q\le n+1+\rho n(n+1)+ \frac{q}{1+(q-2)/(2n)}=n+1+\rho n(n+1)+ \frac{2nq}{q+2n-2}.$$
This is a contradiction. Then $\sharp \mathcal {F}(f,\{H_i\}_{i=1}^q,1)=1$.

Now, if $q>2n+2+2\rho n(n+1)$, we have
$$q>n+1+\rho n(n+1)+ \frac{2nq}{4n}\ge  n+1+\rho n(n+1)+\frac{2nq}{q+2n-2}.$$
The theorem is proved.
\end{proof}

\begin{proof}[{\sc Proof of Theorem \ref{1.2}}]
By Theorem \ref{5.1}, it is enough for us to prove the theorem with $q\le 2n+2+2\rho n(n+1)$.

Suppose contrarily that there exist two distinct elements $f^1,f^2$ of $\mathcal F(f^1,\{H_i\}_{i=1}^{q},n+1)$. 
Similarly as the proof of Theorem \ref{5.1}, we may assume that
$$\underbrace{\dfrac{(f^1,H_1)}{(f^2,H_1)}\equiv \dfrac{(f^1,H_2)}{(f^2,H_2)}\equiv \cdot\cdot\cdot\equiv \dfrac{(f^1,H_{k_1})}
	{(f^2,H_{k_1})}}_{\text { group } 1}\not\equiv
\underbrace{\dfrac{(f^1,H_{k_1+1})}{(f^2,H_{k_1+1})}\equiv \cdot\cdot\cdot\equiv \dfrac{(f^1,H_{k_2})}{(f^2,H_{k_2})}}_{\text { group } 2}$$
$$\not\equiv \underbrace{\dfrac{(f^1,H_{k_2+1})}{(f^2,H_{k_2+1})}\equiv \cdot\cdot\cdot\equiv \dfrac{(f^1,H_{k_3})}{(f^2,H_{k_3})}}_{\text { group } 3}\not\equiv \cdot\cdot\cdot\not\equiv \underbrace{\dfrac{(f^1,H_{k_{s-1}+1})}{(f^2,H_{k_{s-1}+1})}\equiv\cdot\cdot\cdot \equiv 
	\dfrac{(f^1,H_{k_s})}{(f^2,H_{k_s})}}_{\text { group } s},$$
where $k_s=q.$ 

Then, we have
$P_i=(f^1,H_i)(f^2,H_{\sigma (i)})-(f^2,H_i)(f^1,H_{\sigma (i)})\not\equiv 0$, 
for all $1\le i\le q$, where
\begin{equation*}
\sigma (i)=
\begin{cases}
i+n& \text{ if $i+n\leq q$},\\
i+n-q&\text{ if  $i+n> q$}.
\end{cases}
\end{equation*}
For each $1\le i\le q$, we set $S_i=\{z\in \C^m\ |\ \nu_{(f^1,H_i)}(z)\ne\nu_{(f^2,H_i)}(z)\}$.
Then $\overline{S_i}$ is an analytic subset of dimension $m-1$ 
and $\overline{S_i}\setminus S_i$ is an analytic subset of dimension
$\le m-2$. Denote by $\nu_{S_i}$ the reduced divisor with the support $\overline{S_i}$. For $z\in f^{-1}(H_i)$, it is easy to see that:

$\bullet$ If $z\in S_i$ then $\min\{\nu_{(f^1,H_i)}(z),\nu_{(f^2,H_i)}(z)\}>n$. 
Because $\nu_{S_i}(z)=1$, we have
\begin{align*}
\min\{\nu_{(f^1,H_i)}(z),\nu_{(f^2,H_i)}(z)\}\ge n+1\ge \nu_{(f^1,H_i)}^{[n]}(z)+\nu_{S_i}(z).
\end{align*}

$\bullet$ If $z\not\in \overline S_i$ then $\nu_{(f^1,H_i)}(z)=\nu_{(f^2,H_i)}(z)$
and $\nu_{S_i}(z)=0$ then
\begin{align*}
\min\{\nu_{(f^1,H_i)}(z),\nu_{(f^2,H_i)}(z)\}=\nu_{(f^1,H_i)}^{[n]}(z)=\nu_{(f^1,H_i)}^{[n]}(z)+\nu_{S_i}(z).
\end{align*}
It yields that
\begin{align*}
\min\{\nu_{(f^1,H_i)}(z),\nu_{(f^2,H_i)}(z)\}\ge\nu_{(f^1,H_i)}^{[n]}(z)+\nu_{S_i}(z).
\end{align*}
for all $z\in f^{-1}(H_i)$ and hence it holds for all $z\in \mathbb B^m(1)$.

By using the same argument as in the proof of Theorem \ref{1.1}, we have
\begin{align*}
\nu_{P_i}(z) 
\ge&\sum_{v=i,\sigma (i)}\min\{\nu_{(f^1,H_v)}(z),\nu_{(f^2,H_v)}(z)\}+\sum_{\underset{v\ne i,\sigma (i)}{v=1}}^{q}\nu_{(f^1,H_v)}^{[1]}(z)\\
\ge&\sum_{v=i,\sigma (i)}\bigl (\nu_{(f^1,H_v)}^{[n]}(z)+\nu_{S_v}(z)\bigl )+\sum_{\underset{v\ne i,\sigma (i)}{v=1}}^{q}\nu_{(f^1,H_v)}^{[1]}(z).
\end{align*}
Summing-up of both sides of the above inequality over 
$i=1,\ldots,q$, we obtain 
\begin{align}\label{5.3}
\begin{split}
\nu_{\prod_{i=1}^qP_i}(z) \ge &2 \sum_{i=1}^{q}\biggl (\nu_{(f^1,H_i)}^{[n]}(z)+\nu_{S_i}\biggl )+(q-2) \sum_{i=1}^{q}\nu_{(f^1,H_i)}^{[1]}(z)\\
= &\left (1+\frac{q-2}{2n}\right) \sum_{i=1}^{q}\left (\nu_{(f^1,H_i)}^{[n]}(z)+\nu_{(f^2,H_i)}^{[n]}(z)\right )+2 \sum_{i=1}^{q}\nu_{S_i}(z).
\end{split}
\end{align}

Assume that $H_i=\{a_{i0}\omega_0+\cdot\cdot\cdot+a_{in}\omega_n=0\}.$
We set $h_i=\frac {(f^1,H_i)}{(f^2,H_i)}\ (1\le i\le q).$ Then
$\frac {h_i}{h_j} = \frac {(f^1,H_i)\cdot (f^2,H_j)}{(f^1,H_j)\cdot (f^2,H_i)}$
does not depend on representations of $f^1$ and $f^2$ respectively.

Take an arbitrary subset of $2n+2$ elements of the set $\{1,\ldots,q\}$, for instance it is $\{1,\ldots,2n+2\}$.
Since $\sum_{k=0}^n a_{ik}f^{1}_{k}-h_i\cdot \sum_{k=0}^n a_{ik}f^{2}_{k}=0\ (1\le i \le 2n+2),$ it implies that 
$$\det \ (a_{i0},\ldots,a_{in},a_{i0}h_i,\ldots,a_{in}h_i; 1\le i \le 2n+2)=0.$$

For each subset $I\subset \{1,2,\ldots,2n+2\},$ put $h_I=\prod_{i\in I}h_i$. Denote by $\mathcal {I}$ the set of all 
combinations $I=(i_1,\ldots,i_{n+1})$ with $1\le i_1<...<i_{n+1}\le q.$

For each $I=(i_1,\ldots,i_{n+1})\in \mathcal {I}$, define 
$$A_I=(-1)^{\frac {(n+1)(n+2)}{2}+i_1+\cdots+i_{n+1}}\cdot \det (a_{i_rl};1\le r \le n+1,0\le l \le n)\cdot$$
$$\quad\quad\quad\quad\quad\quad\quad \quad\quad\quad\quad\quad\quad\quad\det (a_{j_sl};1\le s \le n+1,0\le l \le n),$$
where $J=(j_1,\ldots,j_{n+1})\in \mathcal {I}$ such that $I \cup J =\{1,2,\ldots,2n+2\}.$
We have 
$$
\sum_{I\in \mathcal {I}}A_Ih_I=0.
$$

Take $I_0\in \mathcal {I}$. Then $A_{I_0}h_{I_0}=-\sum_{I\in \mathcal {I}, I\ne I_0}A_Ih_I$, that is,
$$
h_{I_0}=-\sum_{I\in \mathcal {I}, I\ne I_0}\frac {A_I}{A_{I_0}}h_I.
$$ 
Remark that for each $I\in \mathcal {I}$, then $\frac{A_I}{A_{I_0}}\ne 0.$

Denote by $t$ the minimal number satisfying the following: 
There exist $t$ elements $I_1,\ldots,I_t \in \mathcal {I}\setminus \{I_0\}$ and $t$ nonzero constants $b_i \in 
\C $ such that $h_{I_0}=\sum_{i=1}^tb_ih_{I_i}.$

Since $h_{I_0}\not \equiv 0$ and by the minimality of $t$, 
it follows that the family $\{h_{I_1},\ldots,h_{I_t}\}$ is linearly 
independent over $\C$.

{\bf Case 1.} $t=1$. Then $\dfrac {h_{I_0}}{h_{I_1}}\in\C^*.$

{\bf Case 2.} $t\ge 2.$

Consider the meromorphic mapping $F:\B^m(1) \to \P^{t-1}(\C)$ with a reduced representation
$F=(h_{I_1}h/d:\cdots :d h_{I_t}h/d)$, where $h=\prod_{i=1}^{2n+2}(f^2,H_i)$ and $d$ is holomorphic on $\B^m(1)$.

If $z$ is a zero of $h_{I_i}h/d,$ then $z$ must be either zero or pole of some $h_v$.
Hence $z$ belongs to $S_v$ for some $v$. This yields that
$\nu_{d h_{I_i}}^{[1]}(z)\le\sum_{v=1}^{q}\nu_{S_v}.$

It is clear that $T_F(r,r_0)\le (n+1)(T_{f^1}(r,r_0)+T_{f^2}(r,r_0))-N_d(r,r_0)$. Denote by $W(F)$ the general Wronskian of $F$ and set 
$$G:=\prod_{0\le s<l\le 2}\left (\dfrac{(h_{I_l}h/d-h_{I_s}h/d)\cdot W(F)}{\prod_{i=0}^t(h_{I_i}h/d)}\right ).$$
Then we have $G\in B(0,(t-1)(t+1)/2;F)\subset B(0,(t-1)(t+1)/2;f^1,f^2)$. For each subset $J\subset\{1,\ldots,q\}$, set $J^c=\{1,\ldots,q\}\setminus J$. It is clear that 
$$\bigcup_{0\le s<l\le 2}((I_l\setminus I_s)\cup (I_s\setminus I_l))^c=\{1,\ldots,q\}.$$
We have
\begin{align*}
-\nu_G&=3\sum_{i=0}^t\nu_{h_{I_i}h/d}-3\nu_{W(F)}-\sum_{0\le s<l\le 2}\nu_{h_{I_l}h/d-h_{I_s}h/d}\\
&\le 3\sum_{i=0}^t\nu_{h_{I_i}h/d}^{[1]}- \sum_{0\le s<l\le 2}(\nu^0_{h_{I_l}/h_{I_s}-1})\\
&\le 3\sum_{i=0}^t\nu_{h_{I_i}h/d}^{[1]}- \sum_{0\le s<l\le 2}\sum_{i\in ((I_l\setminus I_s)\cup (I_s\setminus I_l))^c}\nu^{[1]}_{(f,H_i)}\\
&\le 3\sum_{i=0}^t\nu_{h_{I_i}h/d}^{[1]}- \sum_{i=1}^q\nu^{[1]}_{(f,H_i)}.
\end{align*}
Then, we have
\begin{align*}
\nu_{\prod_{i=1}^qP_i} &\ge \left (1+\frac{q-2}{2n}\right) \sum_{i=1}^{q}\left (\nu_{(f^1,H_i)}^{[n]}+\nu_{(f^2,H_i)}^{[n]}\right )+2 \sum_{i=1}^{q}\nu_{S_i}\\
&\ge \left (1+\frac{q-2}{2n}\right) \sum_{i=1}^{q}\left (\nu_{(f^1,H_i)}^{[n]}+\nu_{(f^2,H_i)}^{[n]}\right )+\frac{2}{3(t+1)}\left (\sum_{i=1}^q\nu^{[1]}_{(f,H_i)}+\nu_G\right).
\end{align*}
This yields that
$$ \nu_{G (\prod_{i=1}^qP_i)^{3(t+1)}}\ge  \left(6(t+1)+\frac{1}{n}\right)\sum_{i=1}^{q}\left (\nu_{(f^1,H_i)}^{[n]}+\nu_{(f^2,H_i)}^{[n]}\right).$$
We note that $G (\prod_{i=1}^qP_i)^{3(t+1)}\in B(3q(t+1),(t-1)(t+1)/2;f^1,f^2)$. Hence, from Proposition \ref{3.5}, we have
\begin{align*}
q&\le n+1+2\rho n(n+1)+\dfrac{3q(t+1)+\rho (t-1)(t+1)}{6(t+1)+\frac{1}{n}}\\
&\le n+1+2\rho n(n+1)+\frac{3(t+1)(2n+2+2\rho n(n+1))+\rho (t-1)(t+1)}{6(t+1)+\frac{1}{n}}\\
&=2n+1+\frac{6n(t+1)}{6n(t+1)+1}+\rho \left (2n(n+1)+\frac{6n(n+1)(t+1)+(t-1)(t+1)}{6(t+1)+\frac{1}{n}}\right)\\
&=2n+1+\frac{6np}{6np+1}+\rho \left (2n(n+1)+\frac{6n^2(n+1)p+np(p-2)}{6np+1}\right).
\end{align*}
This is a contradiction. Hence, this case does not happen.

Therefore, for each $I\in \mathcal {I}$, there is $J\in \mathcal {I}\setminus \{I\}$ such that  $\dfrac {h_I}{h_J}\in\C.$

We now consider the torsion free abelian subgroup generated by the family 
$\{[h_1] ,\ldots, [h_{q}]\}$ of the abelian group $\mathcal {M^*}_m/ \C^*.$ 
Then the family $\{[h_1],\ldots,[h_{q}]\}$ has the property $P_{q,n+1}$. 
It implies that there exist $q-2n\ge 2$ elements, without loss of generality 
we may assume that they are  $[h_1],[h_2],$  such that $[h_1]=[h_2].$ Then
$\dfrac {h_1}{h_2}=\lambda\in\C^*.$

Suppose that $\lambda\not\equiv 1.$ Since $\dfrac{h_1(z)}{h_2(z)}=1$ for each $z\in\bigcup_{i=3}^{q}f^{-1}(H_i)\setminus(f^{-1}(H_1)\cup f^{-1}(H_2)),$ 
it implies that  $\bigcup_{i=3}^{q}f^{-1}(H_i)=\emptyset$. Hence $\sum_{i=3}^q\nu_{(f,H_i)}^{[n]}=0$. Then, by Proposition \ref{3.5}, we have
\begin{align*} 
q-2\le n+1+\rho n(n+1).
\end{align*}
This is a contradiction. Thus, $\lambda\equiv 1$, i.e., $h_1\equiv h_2$. Hence 
$\nu_{(f^1,H_i)}=\nu_{(f^2,H_i)},\ i=1,2$.  

Now we consider 
\begin{align*}
P_{1}&=(f^1,H_{1})(f^2,H_{n+1})-(f^1,H_{n+1})(f^2,H_1)\\
&=\frac{(f^1,H_1)}{(f^1,H_2)}\left ((f^1,H_2)(f^{2},H_{n+1})-(f^1,H_{n+1})(f^2,H_2)\right )\not\equiv 0.
\end{align*}
From this inequality, we easily see that
\begin{align}\label{5.4}
\nu_{P_1}\ge (\nu_{(f^1,H_1)}+\nu_{(f^1,H_1)}^{[1]})+\nu_{(f^1,H_{n+1})}^{[n]}+\sum_{\underset{v\ne 1,n+1}{v=1}}^{q}\nu_{(f^1,H_v)}^{[1]}
\end{align}
and similarly
\begin{align}\label{5.5}
\begin{split}
\nu_{P_{q-n+1}}&\ge \nu_{(f^1,H_{q-n+1})}+(\nu_{(f^1,H_1)}^{[1]}+\nu_{(f^1,H_1)}^{[n]})+\sum_{\underset{v\ne 1,q-n+1}{v=1}}^{q}\nu_{(f^1,H_v)}^{[1]},\\
\nu_{P_2}&\ge (\nu_{(f^1,H_2)}+\nu_{(f^1,H_2)}^{[1]})+\nu_{(f^1,H_{n+2})}^{[n]}+\sum_{\underset{v\ne 2,n+2}{v=1}}^{q}\nu_{(f^1,H_v)}^{[1]},\\
\nu_{P_{q-n+2}}&\ge \nu_{(f^1,H_{q-n+2})}+(\nu_{(f^1,H_2)}^{[1]}+\nu_{(f^1,H_2)}^{[n]})+\sum_{\underset{v\ne 2,q-n+2}{v=1}}^{q}\nu_{(f^1,H_v)}^{[1]}.
\end{split}
\end{align}
Then, similar as (\ref{5.3}), we have
\begin{align}\label{5.6}
\nu_{\prod_{i=1}^qP_i}(z) \ge\left (1+\frac{q-2}{2n}\right) \sum_{i=1}^{q}\left (\nu_{(f^1,H_i)}^{[n]}(z)+\nu_{(f^2,H_i)}^{[n]}(z)\right )+\frac{2}{n}(\nu_{(f^1,H_1)}^{[n]}+\nu_{(f^1,H_2)}^{[n]}).
\end{align}

On the other hand, by setting 
$$\gamma (i)=\begin{cases}
i+n&\text{ if }i+n\le q\\
i+n-q+2&\text { if }i+n>q
\end{cases}\text{ and }P'_i=(f^1,H_i)(f^2,H_{\gamma (i)})-(f^2,H_i)(f^1,H_{\gamma (i)}),$$ 
we also have
$$ \nu_{P'_i}\ge\nu^{[n]}_{(f^1,H_i)}+\nu^{[n]}_{(f^1,H_{\gamma (i)})}+\sum_{\overset{j=1}{j\ne i,\gamma (i)}}^q\nu^{[1]}_{(f^1,H_j)}. $$
Summing-up both sides of this inequalities over all $3\le i\le q$, we have
\begin{align}\label{5.7}
\begin{split}
\nu_{\prod_{i=3}^qP'_i}&\ge \sum_{i=3}^{q}\left (2\nu_{(f^1,H_i)}^{[n]}+(q-4)\nu_{(f^1,H_i)}^{[1]}\right)+(q-2)\left (\nu_{(f^1,H_1)}^{[1]}+\nu_{(f^1,H_2)}^{[1]}\right)\\
&\ge \left (1+\frac{q-4}{2n}\right)\sum_{i=3}^{q}(\nu_{(f^1,H_i)}^{[n]}+\nu_{(f^2,H_i)}^{[n]})+\frac{q-2}{n}\left (\nu_{(f^1,H_1)}^{[n]}+\nu_{(f^1,H_2)}^{[n]}\right)\\
&=\left(1+\frac{q-4}{2n}\right)\sum_{i=1}^{q}(\nu_{(f^1,H_i)}^{[n]}+\nu_{(f^2,H_i)}^{[n]})-\frac{2n-2}{n}\left (\nu_{(f^1,H_1)}^{[n]}+\nu_{(f^1,H_2)}^{[n]}\right)
\end{split}
\end{align}
From (\ref{5.6}) and (\ref{5.7}), we have 
\begin{align*}
\nu_{(\prod_{i=1}^qP_i)^{n-1}\cdot (\prod_{i=3}^qP_i')}&\ge \left(n-1+\frac{(n-1)(q-2)}{2n}+1+\frac{q-4}{2n}\right) \sum_{i=1}^{q}(\nu_{(f^1,H_i)}^{[n]}+\nu_{(f^2,H_i)}^{[n]})\\
&=\frac{nq+2n^2-2n-2}{2n}\sum_{i=1}^{q}(\nu_{(f^1,H_i)}^{[n]}+\nu_{(f^2,H_i)}^{[n]}).
\end{align*}
It is clear that $(\prod_{i=1}^qP_i)^{n-1}\cdot (\prod_{i=3}^qP_i')\in B(nq-2,0;f^1,f^2)$. Then from Lemma \ref{3.5}, we have
\begin{align*}
q&\le n+1+\rho n(n+1)+ \frac{(n-1)q+q-2}{n-1+\frac{(n-1)(q-2)}{2n}+1+\frac{q-4}{2n}}\\
&\le n+1+\rho n(n+1)+\frac{2n(n(2n+2+2\rho n(n+1))-2)}{n(2n+2)+2n^2-2n-2}\\
&=n+1+\rho n(n+1)+\frac{4n^3+4n^2-4n+4\rho n^2(n+1)}{4n^2-2}\\
&\le 2n+1+\frac{4n^2-2n}{4n^2-2}+\rho(n(n+1)+\frac{4n^2(n+1)}{4n^2-2}).
\end{align*}
This is a contradiction. 

Hence $f^1\equiv f^2$. The theorem is proved. 
\end{proof}

\section*{Acknowledgements} This work was done during a stay of the author at Vietnam Institute for Advanced Study in Mathematics (VIASM). He would like to thank the institute for the support. This research is funded by Vietnam National Foundation for Science and Technology Development (NAFOSTED) under grant number 101.04-2018.01.

\end{document}